\documentclass[a4paper,11pt]{amsart}
\usepackage{amsmath, amsthm, amscd, amssymb, amsfonts, amsxtra, amssymb, latexsym}
\usepackage{enumerate}
\usepackage[utf8x]{inputenc}
\usepackage{hyperref}
\hypersetup{colorlinks = true,	allcolors  = blue}
\usepackage{float}

\newcommand{\Z}{\mathbb{Z}}
\newcommand{\N}{\mathbb{N}}
\newcommand{\ff}{\mathbb{F}}
\newcommand{\Tr}{\operatorname{Tr}}

\newcommand{\CC}{\mathcal{C}}
\newcommand{\Cay}{\mathrm{Cay}}

\newcommand{\G}{\Gamma}

\newcommand{\sk}{\smallskip}
\newcommand{\msk}{\medskip}

\newtheorem{thm}{Theorem}[section]
\newtheorem{prop}[thm]{Proposition}
\newtheorem{lem}[thm]{Lemma}
\newtheorem{coro}[thm]{Corollary}

\theoremstyle{definition}
\newtheorem{rem}[thm]{Remark}
\newtheorem{exam}[thm]{Example}

\theoremstyle{remark}

\usepackage{color}

\def\black{\color{black}}

\begin{document} \sloppy
\numberwithin{equation}{section}
\title[A reduction formula for Waring numbers]{A reduction formula for Waring numbers \\ 
	through generalized Paley graphs}
\author{Ricardo A.\@ Podest\'a, Denis E.\@ Videla}
\dedicatory{\today}
\thanks{2020 {\it Math.\@ Subject Class.\@} Primary 11P05;\, Secondary 05C25, 05C12, 11A07.}
\thanks{Partially supported by CONICET, FONCyT and SECyT-UNC}

\address{Ricardo A.\@ Podest\'a. FaMAF -- CIEM (CONICET), Universidad Nacional de C\'ordoba. 
	Av.\@ Medina Allende 2144, Ciudad Universitaria, (5000) C\'ordoba, Rep\'ublica Argentina. 
	{\it E-mail: podesta@famaf.unc.edu.ar}}
\address{Denis E.\@ Videla. FaMAF -- CIEM (CONICET), Universidad Nacional de C\'ordoba. 
	Av.\@ Medina Allende 2144, Ciudad Universitaria, (5000) C\' ordoba, Rep\'ublica Argentina. 
	{\it E-mail: devidela@famaf.unc.edu.ar}}

\begin{abstract}
We give a reduction formula for the Waring number $g(k,q)$ over a finite field $\mathbb{F}_q$. 
By exploiting the relation between $g(k,q)$ with the diameter of the generalized Paley graph $\G(k,q)$ and by using the characterization due to Pearce and Praeger (2019) of those $\G(k,q)$ which are Cartesian decomposable, we obtain the reduction formula 
	$$g(\tfrac{p^{ab}-1}{bc},p^{ab}) = b g(\tfrac{p^a-1}{c},p^a)$$
for $p$ prime and $a,b,c$ positive integers under certain arithmetic conditions. 
Then, we find some arithmetic conditions 
to apply the formula above, which allow us to obtain many infinite families of explicit values of Waring numbers.
Finally, we use the reduction formula 
together with the characterization of $2$-weight irreducible cyclic codes 
due to Schmidt and White (2002) to find infinite families of explicit even values of $g(k,q)$. 
\end{abstract}

\maketitle

\section{Introduction}
This work is a natural continuation of \cite{PV3} where we begin the study of Waring numbers over finite fields by seeing them as the diameter of generalized Paley graphs. Here, we will obtain a general reduction formula for Waring numbers over finite fields  by using Cartesian decomposable generalized Paley graphs. Up to our best knowledge, there are no such kind of reduction formulas 
available in the literature.

\subsubsection*{Waring numbers}
The classical problem introduced by Waring as far back as 1770 deals with positive integers and asks whether, given any number $k$, there is a number $g(k)$ such that every positive integer can be expressed as a sum of at most $g(k)$ $k$-th powers. 
A similar problem can be considered in the context of finite fields. Given a finite field $\ff_q$ of $q$ elements and a positive integer $k$, the question is to decide if every element $x\in \ff_q$ can be written as a sum of $k$-th powers in the field, that is $x=x_1^k+\cdots+x_s^k$ 
for some $x_i\in \ff_q$ for $i=1,\ldots,s$. In this case, the Waring number $g(k,q)$ is defined to be the minimal value $s$ (if it exists) such that every element of $\ff_q$ is a sum of a number $s$ of $k$-th powers in $\ff_q$.  

First, notice that if $q=p^m$ we have 
\begin{equation} \label{existence}
g(k,q) \text{ exists } \qquad \Leftrightarrow \qquad \tfrac{p^m-1}{p^d-1} \nmid k \quad 
\text{for every $d\mid m$ with $d\ne m$.}
\end{equation}
Since $x^q=x$ for every element $x \in \ff_q$, we trivially have that $g(1,q)=g(q,q)=1$ and it is only necessary to consider $2 \le k \le q-1$. 
Also, it is not difficult to see that 
$g(k,q)=g(\gcd(k,q-1),q)$ and, thus, we will always assume that $k\mid q-1$ (if $(k,q-1)=1$ then $g(k,q)=g(1,q)=1$).

The study of Waring numbers in finite fields was already mentioned by Hardy and Littlewood in 1928 (\cite{HL}, see also \cite{To}) and can be probably traced back to Cauchy. 
The modern study of Waring numbers was probably initiated by Dodson and Tietäväinen (\cite{Do}, \cite{DT}, \cite{Ti}, 1968--1976) and Small (\cite{Sm}, 1977) in the seventies. 
After them, several authors studied these numbers; see for instance the works of Cipra, Cochrane and Pinner (\cite{Ci}, \cite{CiCP}, \cite{CP}), Garc\'ia-Sol\'e (\cite{GS}), Glibichuk et al.\@ (\cite{Gl}, \cite{GlR}), Kononen (\cite{KK}), Konyagin (\cite{Kon}), Moreno-Castro (\cite{MC1}, \cite{MC}), Podest\'a-Videla (\cite{PV3}) and Winterhof et al.\@ (\cite{Win}, \cite{W}, \cite{Win2}). 
In these works several (mainly upper) bounds for the numbers $g(k,q)$ and some (few) exact values are given. Brief compilations of these results can be found in \S 7.3.4 of the handbook \cite{MP} or in  \S 2 of \cite{PV3}.

\subsubsection*{Generalized Paley graphs}
For a graph $\G$ we will denote by $V(\G)$ and $E(\G)$ its sets of vertices and edges, respectively. 
An (undirected) edge from $x$ to $y$  will be denoted by $xy$ while a directed edge (or arc) from $x$ to $y$ by $\vec{xy}$.
If $G$ is a group and $S\subset G$, the Cayley graph $\G=\Cay(G,S)$ is the directed graph (digraph) with vertex set $G$ and $\vec{xy} \in E(\G)$ if and only if $y-x \in S$.
As usual, if both $x-y$ and $y-x$ are in $S$, then the two arcs $\vec{xy}$ and $\vec{yx}$ are considered as a single undirected edge $xy$.

Let $q=p^m$ with $p$ a prime number and $k$ a non-negative integer with $k\mid q-1$, 
the \textit{generalized Paley graph} (\textit{GP-graph} for short) is the Cayley graph
\begin{equation} \label{Gammas}
\G(k,q) = \Cay(\ff_{q},R_{k}) \quad \text{with } \quad R_{k} = \{ x^{k} : x \in \ff_{q}^*\}.
\end{equation} 
That is, $\G(k,q)$ is the directed graph with vertex set $\ff_{q}$ and two vertices $u,v \in \ff_{q}$ form an arc 
from $u$ to $v$ if and only if 
$v-u=x^k$ for some $x\in \ff_{q}^*$. 
Since we exclude $0$ from the vertex set, there are no loops in these graphs.
Notice that $\G(k,q)$ is an $n$-regular graph with $n=\tfrac{q-1}k$. 
The graph $\G(k,q)$ is undirected either if $q$ is even or if $k$ divides $\tfrac{q-1}2$ when $p$ is odd (equivalently if $n$ is even when $p$ is odd) and it is connected if $n$ is a primitive divisor (see \eqref{dagger}) of $q-1$.  
When $k=1$ we get the complete graph $\G(1,q)=K_q$ and when $k=2$ we get the classic Paley graph $\Gamma(2,q) = P(q)$. 

If the GP-graph $\G=\G(k,q)$ is connected (not necessarily undirected), 
the diameter of $\G$ is the Waring number of the pair $(k,q)$, that is we have 
(see Theorem 3.3 in \cite{PV3}, see also \cite{GS} for $q=p$ prime)
\begin{equation} \label{g=d} 
g(k,q) = \delta(\G(k,q)),
\end{equation}
where $\delta(\G)$ is the diameter of $\G$.

Generalized Paley graphs have been extensively studied in the few past years. 
Lim and Praeger studied their automorphism groups and characterized all GP-graphs which are Hamming graphs 
(\cite{LP}). 
Also, Pearce and Praeger characterized all GP-graphs which are cartesian decomposable (\cite{PP}).
Both classic Paley graphs and generalized Paley graphs have been used to find linear codes with good decoding properties (\cite{GK}, \cite{KL}, \cite{SL}). 
They can also be seen as particular regular maps in Riemann surfaces 
(\cite{J}, \cite{JW}). 
The number of walks in GP-graphs are related with the number of solutions of 
diagonal equations over finite fields (\cite{V}). 
Under some mild restrictions, the spectrum of GP-graphs determines the weight distribution of their associated irreducible
codes (\cite{PV2}, \cite{PV4}).

\subsubsection*{Outline and results}
In Section 2 we consider Cartesian decomposable GP-graphs and obtain general reduction formulas for Waring numbers of the form $g(k,p^{ab})=bg(u,p^a)$ with $p$ prime. In Sections 3 and 5 we get more explicit expressions for the reduction formula depending whether $b$ is a prime, a prime power or an arbitrary integer. In Sections 6 and 7 we study even Waring numbers in more detail.

Now, we summarize the results of the paper more precisely.
In Section~2, we consider GP-graphs $\G(k,q)$ which are Cartesian decomposable. 
To get a full result we have to consider both directed and undirected graphs, treating these cases separately.  
In Proposition~\ref{dir prod PP} we extend the Pearce and Praeger's characterization of undirected Cartesian decomposable GP-graphs in \cite{PP} to directed Cartesian decomposable GP-graphs. By using these characterizations and \eqref{g=d}, in Theorem \ref{waringprod main} we get a reduction formula for the associated Waring numbers $g(k,q)$ where $q=p^{ab}$ with $p$ prime. 
Namely,  if $a,b,c$ are integers such that $c$ and $bc$ primitive  divisors of $p^a-1$ and $p^{ab}-1$ respectively, we get 
\begin{equation} \label{red fla}
g(\tfrac{p^{ab}-1}{bc},p^{ab}) = b  g(\tfrac{p^a-1}{c}, p^{a}).
\end{equation}

In Section 3, we consider the particular case that $b$ is prime. The general case is treated in Proposition~\ref{prod primecase}) were the arithmetic conditions on $a,b,c$ for the validity of the reduction formula are easier. 
As a consequence, we obtain several particular instances of reduction formula for Waring numbers in Corollaries \ref{coro2}, \ref{coro1}, \ref{coro3}, \ref{coro4} and \ref{coro5}; as well as some general examples (Examples \ref{otro ejemplo} and \ref{g=4}).

In the next section, we study the divisibility conditions in Theorem \ref{waringprod main} in terms of divisibility 
properties of the polynomial $\Psi_b (x) = x^{b-1} + \cdots + x +1$.
This section is in some sense elementary, but somewhat technical and can be skipped at first reading. Here we provide the arithmetic tools to study reduction formulas in the case of more general integers $b$.
In Section 5, we apply these divisibility conditions to provide more explicit reduction formulas for $g(k,q)$. 
We consider the case of $b=r^t$ a prime power in Proposition~\ref{red pow} and the case of $b=2^t s$ with $s$ odd in Theorem~\ref{teo 42}. These results together cover all possible cases of $b\in \N$.
As an a priori unexpected consequence, we then reobtain Kononen's expression \eqref{kononen1} and \eqref{kononen2}, with $m > 1$, as a special case of our reduction formulas and the basis case $m=1$ previously proved by Winterhof and van de Woestijne \cite{Win2}.

Finally, in the last two sections we study even Waring numbers in more detail. To this end, we use a relation between GP-graphs and 2-weight irreducible cyclic codes. 
First, in Section 6 we focus on Waring numbers equal to $2$. 
Among the known cases of $g(k,q)=2$ from other works (see List 1 at the beginning of Section 6) we provide new pairs $(k,q)$ such that $g(k,q)=2$. 
To do this, we consider strongly regular GP-graphs $\Gamma(k,q)$. These graphs are determined by the spectra of 2-weight irreducible cyclic codes $\mathcal{C}(k,q)$ satisfying $k\mid \frac{q-1}{p-1}$ (see \cite{PV2}). In \cite{SW}, Schmidt and White conjectured that there are only three disjoint families of this kind of codes (subfield subcodes, semiprimitive and exceptional) and therefore three associated disjoint families of strongly regular GP-graphs.
Since connected strongly regular GP-graphs have diameter $2$ we get that $g(k,q)=2$ for those $k$ and $q$ corresponding to these graphs (see Proposition \ref{Waringsrg}). In Theorem \ref{srg waring} we proved that for every pair $(k,q)$ associated to a semiprimitive or exceptional code $\mathcal{C}(k,q)$ we have $g(k,q)=2$. By identifying a family of semiprimitive pairs $(k,q)$ in Proposition \ref{otra prop} we obtain an infinite number of Waring numbers equal to 2 which cannot be obtained by Small condition (see List 1).  

Secondly, in Section 7, we first recollect known cases of even Waring numbers (see Lists~2 and 3). Then, we consider decomposable GP-graphs $\G(k,p^{ab}) = \square^b \G(u,p^b)$ where $(u,p^b)$ is either a semiprimitive pair or $2\le k\le q^{\frac a4}+1$. Thus, 
by combining the reduction formula \eqref{red fla} with the fact that $g(u,p^b)=2$ in this cases, we get a general result for even Waring numbers in Proposition~\ref{teo par}. As a consequence, a formula of the form 
$g(\tfrac{p^{ab}-1}{bc},p^{ab})=2b$ is obtained, for some integers $a,b,c$ under proper easy arithmetic conditions.
We end this work with a compilation of Waring numbers equal to 4 (see List 4), a finite number from known results, and infinite families obtained from the results in the section.

\section{A reduction formula for Waring numbers}
By studying the diameter of Cartesian decomposable GP-graphs, we will obtain a general reduction formula for Waring numbers, expressing the value $g(k_b, q^{b})$ as $bg(k,q)$ where $q=p^a$ with $p$ prime and $k<k_b$, for certain integers $a,b,k,k_b$ satisfying arithmetic conditions. We will have to consider directed and undirected graphs separately in order the get the full result.

The Cartesian product of undirected graphs $\Gamma_1,\ldots,\Gamma_b$, denoted as 
$\Gamma_1 \, \square \cdots \square \, \Gamma_b$,  
is the graph $\Gamma$ with vertex set $V(\Gamma) = V(\Gamma_1)\times\cdots\times V(\Gamma_b)$, such that if $v=(v_1,\ldots,v_{b})$, $w=(w_1,\ldots,w_b) \in V(\G)$ then $v$ and $w$ form an edge in $\G$ if and only if there exists only one $j\in\{1,\ldots,b\}$ such that $v_jw_j$ is an edge in $\Gamma_j$ and $v_i=w_i$ for all $i\neq j$. 
An undirected graph $\Gamma$ is called \textit{Cartesian decomposable} if it can be decomposed as 
	$$\G= \Gamma_1 \, \square \cdots \square \, \Gamma_b$$ 
for some undirected graphs $\Gamma_1,\ldots,\Gamma_b$ where $b>1$ and each $\G_i$ has at least two vertices. 
Similarly, if $\Gamma_1,\ldots,\Gamma_b$ are directed graphs one can define their directed Cartesian product, denoted 
	$$\Gamma_1 \, \vec{\square} \cdots \vec{\square} \, \Gamma_b,$$
simply by replacing the word \textit{edge} by the word \textit{arc} 
in the definition of undirected Cartesian product.
A directed graph $\Gamma$ is called Cartesian decomposable if it can be decomposed as $\G= \Gamma_1 \, \vec{\square} \cdots \vec{\square} \, \Gamma_b$ for some $b$ and some directed graphs $\Gamma_1,\ldots,\Gamma_b$.

Let $\Gamma$ be an undirected GP-graph which is Cartesian decomposable.
By the characterization given by Pearce and Praeger in 2019, $\Gamma$ is the Cartesian product of copies of a single GP-graph. 
More precisely, if $\Gamma$ is undirected and connected the following conditions are equivalent (see \cite{PP}): 
\vspace{1mm}
\begin{equation} \label{cond desc}
\begin{split}
& (a) \:\: \Gamma = \G(k,p^m) \text{ is Cartesian decomposable}. \\[1mm]
& (b) \:\: \tfrac{p^m -1}{k}=bc \text{ with $b>1$, $b\mid m$ and $c$ is a primitive divisor of $p^{\frac{m}{b}}-1$}. \\[1mm]
& (c) \:\: \Gamma \cong \square^{b} \Gamma_0, \text{ where $\Gamma_{0}= \G(\tfrac{p^{\frac mb}-1}{c}, p^{\frac mb})$ for  $b,c$ as in $(b)$}.
\end{split}
\end{equation}

\noindent Recall that an integer $e$ is a \textit{primitive divisor} of $p^a-1$ if, by definition, $e \mid p^a-1$ and $e \nmid p^t-1$ for any $t<a$. 
For simplicity, we denote this fact by 
\begin{equation} \label{dagger} 
e \dagger p^a-1.
\end{equation}
Notice that we are using a different notation than Pearce and Praeger. In their work, Pearce and Praeger use $d$ instead of our $k$ and $k$ instead of our $\frac{p^m-1}{k}$.

The paradigm of Cartesian decomposable graphs is given by the Hamming graphs $H(b,q)$. If $V$ is any set with $q$ elements, $H(b,q)$ has vertex set $V^b= V\times \cdots \times V$ ($b$-times) and two vertices form and edge if they are at Hamming distance one (i.e., they differ in a single coordinate). Note that $H(b,q)$ has size $q^b$ and is $b(q-1)$-regular.
Basic examples of Hamming graphs are: the generalized quadrangle $GQ(2,1)=H(2,3)$, the lattice graph $L_{q,q}=H(2,q)$ and the hypercubes $Q_d=H(d,2)$.   
It is easy to see from the definition that Hamming graphs are Cartesian products of the same complete graph, namely 
	$$H(b,q)=\square^b K_q.$$ 

GP-graphs $\G$ which are Hamming graphs are considered in \cite{LP} and \cite{PV3}. 
In this case, if $\G$ is a Hamming graph then 
$$\G=\G(k,q^b)=H(b,q)$$ 
with $q=p^{a}$. 
Then $\G= \square^b \Gamma_0$ where $\G_0=\G(1,p^a)=K_{p^a}$ is the complete graph (i.e., corresponds to the case $u=1$ in \eqref{cond desc}). 
Equating the regularity degrees of $\G$ and $\square^b \G_0$ we have that
	$$\tfrac{q^b-1}{k}=b(q-1).$$ 
Hence, in particular, both $k$ and $b$ divide $\tfrac{q^b-1}{q-1}= q^{b-1}+ \cdots + q+1$.

We now give our first version of the reduction formula for Waring numbers, using undirected GP-graphs.
\begin{prop} \label{waringprod}
Let $p$ be a prime and let $a, b, c \in \N$ such that $c \dagger p^a-1$ 
and $bc \dagger p^{ab}-1$. If $p$ is even or else if $p$ is odd and $bc$ is even then we have 
\begin{equation} \label{gk=bgu}
g(\tfrac{p^{ab}-1}{bc},p^{ab}) = b  g(\tfrac{p^a-1}{c}, p^{a}).
\end{equation} 
\end{prop}

\begin{proof}
First note that \eqref{gk=bgu} trivially holds for $b=1$. Now, if $b>1$, consider the graph 
$\G=\G(\tfrac{p^{ab}-1}{bc},p^{ab})$. Since $bc\dagger p^{ab}-1$ and $p$ is even or $p$ is odd and $bc$ is even, 
$\G$ is connected and undirected. 
Thus, by the characterization in \eqref{cond desc},
we have that $\Gamma = \square^{b} \Gamma_0$ with 
$$\Gamma_{0} = \G(\tfrac{p^{a}-1}{c},p^{a})$$ 
It is clear that $\G_0$ must be undirected also, for if not $\G$ would be directed by the very definition of the directed Cartesian product.

It is well-known (\cite[Proposition 5.1]{HIK}) that if $\Gamma = \Gamma_{1} \Box \cdots \Box \,\Gamma_b$ and $x,y\in V(\Gamma)$, then 
the distance between $x$ and $y$ is given by the sum of the distances between the projections of $x$ and $y$ to each $\G_i$. That is, 
$$ d_{\Gamma}(x,y) = \sum_{1\le i \le b} d_{\Gamma_i} (x_i,y_i),$$
where $x_i$ and $y_i$ denote the $i$-th coordinates of $x$ and $y$ respectively.
Thus, the diameter of $\Gamma$ is $b$ times the diameter of $\Gamma_0$, i.e. 
\begin{equation} \label{diam}
	\delta(\Gamma) = \delta(\square^{b}\Gamma_0) = b \delta(\Gamma_0).
\end{equation}

We saw that $\G$ is connected and since $c \dagger p^a-1$, the graph $\Gamma_0$ is also connected. 
Hence, the Waring numbers $g(\tfrac{p^{a}-1}{c},p^a)$ and $g(\tfrac{p^{ab}-1}{bc},p^{ab})$ both exist. Also, by \eqref{g=d},  $g(\tfrac{p^{a}-1}{c},p^a)$ and $g(\tfrac{p^{ab}-1}{bc},p^{ab})$ equal the diameter of $\Gamma(\tfrac{p^{a}-1}{c},p^a)$ and $\Gamma(\tfrac{p^{ab}-1}{bc},p^{ab})$, respectively. 

Therefore, from \eqref{diam} we get $g(\tfrac{p^{ab}-1}{bc},p^{ab}) = b g(\tfrac{p^{a}-1}{c}, p^{a})$, as desired. 
\end{proof}

We now focus on directed Cartesian products of graphs. This will allow to cover the remaining case $bc$ odd when $p$ is odd 
in the previous proposition. 

We begin by giving necessary conditions for a Caley graph to have a directed Cartesian decomposition 
(the directed equivalent of Lemma 2.4 in \cite{PP}).

\begin{lem} \label{lem Cay dig}
Let $G$ be a finite group such that $G=H_1\times\cdots \times H_b$ with $H_1, \ldots,H_b$ subgroups of $G$. 
Suppose there are subsets $S_i \subset H_i$ such that the identity $e_G \not \in S_i$ and 
$S_i\neq S_{i}^{-1}$ for each $1\le i\le b$.
If $\G=\Cay (G,S)$ with $S=\cup_{i=1}^b S_i$, then $\G \cong \vec{\square}_{i=1}^{\, b}\G_i$ where $\G_i=\Cay(H_i, S_i)$.
That is,  
	$$\Cay (H_1\times\cdots \times H_b, S_1 \cup \cdots \cup S_b) \cong 
		\Cay(H_1, S_1) \vec{\square} \cdots \vec{\square} \Cay(H_b, S_b).$$
\end{lem}

\begin{proof}
For each $i=1,\ldots, b$ consider the graph $\G_i=\Cay(H_i,S_i)$. 
Conditions $e_G \notin S_i$ and $S_i\ne S_{i}^{-1}$ ensure that $\G_i$ is a directed graph without loops for all $i=1,\ldots,b$.

Since $G = H_1\times\cdots \times H_b$, the vertex set of $\G$ is  
$V(\G) = V(\G_1)\times \cdots \times V(\G_b)$.
Let $g = (g_1,\ldots,g_b)$, $h = (h_1,\ldots ,h_b)\in G$. 
By the definition of $\Cay(G,S)$, $\vec{gh} \in E(\G)$ if and only if $g^{-1}h \in S$, 
and this happens if and only if $ g^{-1}h\in S_j$ for some $j$, since $S=S_1 \cup \cdots \cup S_b$.
The last condition is equivalent to 
$g_{j}^{-1}h_j \in S_j$ and $g_i= h_i$ for all $i\ne j$. 
Finally this holds if and only if $\vec{gh}$ is an arc in 
$\vec{\square}_{i=1}^{\, b} \G_i$, by the definition of the
directed Cartesian product. But this shows that $\G=\vec{\square}_{i=1}^{\, b} \G_i$ as we wanted to see.
\end{proof}

Now, applying the previous lemma for directed generalized Paley graphs, we get arithmetic conditions for directed Cartesian decomposability.

\begin{prop} \label{dir prod PP}
Let $p$ be an odd prime and let $a, b, c$ be positive integers such that $c \dagger p^a-1$ and $bc \dagger p^{ab}-1$. 
If $bc$ is odd then  
	$$\G(\tfrac{p^{ab}-1}{bc}, p^{ab}) = \vec{\square}^{\, b} \G(\tfrac{p^a-1}{c},p^a).$$
\end{prop}

\begin{proof}
Consider the graph $\G=\G(k,q)=\Cay(\ff_{q}, S)$ with $k=\tfrac{p^{ab}-1}{bc}$, $q=p^{ab}$ and 
$S=R_{k} = \{x^k : x\in \ff_{q}^* \}$.  
By the comments in the Introduction after \eqref{Gammas} we have that $\G$ is directed, since by hypothesis both $p$ and $n=\frac{q-1}{k}=bc$ are odd, and that $|S|=\tfrac{p^{ab}-1}{k} = bc$.

Let $\omega$ be a primitive element of $\ff_{p^{ab}}$, 
let $C$ be the multiplicative cyclic subgroup of $\ff_{p^{ab}}^*$ generated 
by $\omega^t$ with $t=\tfrac{p^{ab}-1}{c}$, that is $C=R_{t}=\langle \omega ^{t} \rangle$, and consider the cosets 
$S_i= C \omega^{ki}$ for $i=1,\ldots,b$. 
Then we have $S_b=C$ and $|S_i|=|C|=c$  for all $i=1,\ldots,b$.
From the proof of Lemma 3.2 in \cite{PP} we have that
$$S=\bigcup_{1\le i \le b} S_i \qquad \text{ and } \qquad \ff_{p^{ab}} = \bigoplus_{1\le i \le b} \langle S_i \rangle^+$$
where $\langle S_i \rangle^+$ denotes the additive subgroup generated by $S_i$ in the field.
By Lemma~3.1 in \cite{PP}, $\langle S_{b}\rangle ^{+}=\langle C \rangle ^{+}$ is a field. 
On the other hand $|\langle S_i\rangle ^{+}|=p^a$ for all $i=1,\ldots, b$ (see the proof of Lemma 3.2 in \cite{PP}).
In particular, we have that 
\begin{equation} \label{CayS+}
	\Cay (\langle S_{b}\rangle^{+},S_b)= \G(\tfrac{p^a -1}{c},p^{a}),
\end{equation}
since $S_b$ is a multiplicative subgroup of $\ff_{p^a}^*$ or order $c$.
	
Now, since $bc$ is odd by hypothesis, $c$ is odd and hence $S_i \ne -S_{i}$ for all $i=1,\ldots ,b$.
Indeed, if $S_i= -S_i$ for some $i$, after multiplying by $\omega^{k(b-i)}$ we get that
$S_b=-S_b$, i.e.\@ $R_u=-R_u$. But this can only happen if and only if $c$ is even, 
since $p$ is odd. Also, the additive identity of the field $0 \notin S_i$ for every $i=1,\ldots,b$, since each $S_i$ is a non-zero multiplicative coset of the multiplicative subgroup $R_t$ of $\ff_{p^{ab}}^{*}$.
	
By taking, $\G_{i}=\Cay (\langle S_{i}\rangle^{+},S_i)$, the group automorphism  
$m_{\omega^{ki}}: \ff_{p^{ab}}^+\rightarrow \ff_{p^{ab}}^+$ 
given by $m_{\omega^{ki}}(x)=x \omega^{ki}$, maps  
$\langle S_b\rangle ^{+}$ onto $\langle S_i \rangle ^{+}$ and $S_b$ onto $S_i$,
so induces a directed graph isomorphism from $\G_{b}$ to $\G_i$. 
Hence 
$\G_{i}\cong \G(\tfrac{p^a-1}{c},p^a)$
for all $i=1,\ldots,b$.
Finally, by \eqref{CayS+} and Lemma \ref{lem Cay dig}, we have that 
	$$\G(\tfrac{p^{ab-1}}{bc}, p^{ab}) \cong \vec{\square}_{i=1}^{\, b} \G_i \cong \vec{\square}^{\, b} \G(\tfrac{p^a-1}{c},p^a),$$
as asserted.
\end{proof}

As a direct consequence we obtain our first main result, a general reduction formula for Waring numbers. 

\begin{thm} \label{waringprod main}
Let $p$ be a prime and $a, b, c$ positive integers such that $c \dagger p^a-1$ and $bc \dagger p^{ab}-1$. Then, we have 
 \begin{equation} \label{gk=bgu gen}
  g(\tfrac{p^{ab}-1}{bc},p^{ab}) = b g(\tfrac{p^a-1}{c}, p^{a}).
 \end{equation} 
\end{thm}

\begin{proof}
If $p$ is even or if $p$ is odd and $bc$ is even the result is just Proposition~\ref{waringprod}. So, assume that both $p$ and $bc$ are odd and that $b>1$ (the case $b=1$ being trivial). 
Consider the graph $\G(\tfrac{p^{ab-1}}{bc}, p^{ab})$ which, by our assumptions, is directed.
By the hypothesis on $c$ and $bc$ we can apply Proposition \ref{dir prod PP} and hence we have that 
	$$\G(\tfrac{p^{ab-1}}{bc}, p^{ab}) = \vec{\square}^{\, b} \G(\tfrac{p^a-1}{c},p^a).$$

In the same way as for the usual (undirected) Cartesian product, one can prove that 
$	\delta (\vec{\square}_{i=1}^{\, b} \G_{i}) = \sum_{i=1}^b \delta (\Gamma_i)$
for directed graphs $\G_1,\ldots, \G_b$.
In this way, we get that 
$$\delta \big( \G(\tfrac{p^{ab-1}}{bc}, p^{ab}) \big) = \delta \big( \vec{\square}^{\, b} \G(\tfrac{p^a-1}{c},p^a) \big) 
= b \, \delta \big( \G(\tfrac{p^a-1}{c},p^a) \big).$$
The result now follows directly from \eqref{g=d}. 
\end{proof}

\begin{rem}
($i$) The case $\tfrac{p^a-1}{c}=1$, that is 
$g(\tfrac{p^{ab}-1}{b(p^a-1)}, p^{ab}) = b$, 
was the main topic of our previous work \cite{PV3}, where we considered GP-graphs $\G(k,q)$ being also 
Hamming graphs. In fact, $\G(1,p^a)=K_{p^a}$ is the complete graph of $p^{a}$ vertices and, hence, $\G(\tfrac{p^{ab}-1}{b(p^a-1)},p^{ab})=\square^b \G(1,p^a)$ is the Hamming graph $H(b, p^a)$.

\noindent ($ii$) 
To our best knowledge, there is no other exact reduction formula of this kind in the literature. There is, however, a similar upper bound due to Winterhof (\cite{Win}) asserting that  
	\begin{equation} \label{ineq red}
		g(\tfrac{p^m-1}{n},p^m) \le m\, g \big (\tfrac{p-1}{(n,p-1)},p \big ).
	\end{equation}	
 We compare the reduction formula \eqref{gk=bgu gen} with the reduction inequality \eqref{ineq red}.
 Clearly, the only possibility for the parameters in \eqref{ineq red} to be as in  
 \eqref{gk=bgu gen} is the following:
 	$$a=1, \qquad b=m, \qquad c=(n,p-1) \qquad \text{and} \qquad n=bc.$$ 
 In this way, for the parameters above, equality holds in Winterhof's inequality and turns out to be a special instance of \eqref{gk=bgu gen} in this case.
\end{rem}

\section{The reduction formula when $b$ is prime}
In this section we restrict ourselves to the case that $b$ in Theorem \ref{waringprod main} is a prime number and give arithmetic conditions to have the corresponding reduction formula. By studying these conditions we then obtain some exact formulas or reduction formulas for Waring numbers of the form $g(k,p^{ab})$ with $b$ prime.

From now on, for any positive integer $t$ we fix the notation 
\begin{equation*} \label{Psi b}
\Psi_t (x) =  x^{t-1} + \cdots + x +1.
\end{equation*} 
Note that for $x\ne 1$ we have $\Psi_t(x) = \tfrac{x^t-1}{x-1}$.
  
\begin{rem} \label{alternative expression}
	Suppose that $p, a, b , c$ are integers as in Theorem \ref{waringprod main}. If we put $u=\frac{p^a-1}{c}$ and $k=\frac{p^{ab}-1}{bc}$, 
	then the formula \eqref{gk=bgu} takes the form 
	\begin{equation*} \label{g con u}
		g(\tfrac{u}{b} \Psi_{b}(p^a),p^{ab}) = b g(u,p^a),
	\end{equation*}
	since we have $\tfrac{u}{b} \Psi_{b}(p^a) =  \tfrac{u(p^{ab}-1)}{b(p^a-1)} = \tfrac{p^{ab}-1}{bc} =  k$.
\end{rem}

In the next two lemmas, given that $c$ is a primitive divisor of $p^a-1$, we give necessary and sufficient conditions for $bc$ 
to be a primitive divisor of $p^{ab}-1$.

\begin{lem} \label{red primitive}
	Let $p$ be a prime, let $a,b,c$ be positive integers such that $c \dagger p^a-1$. 
	Then, $ bc \dagger p^{ab}-1$ if and only one of the following equivalent conditions hold:
\begin{enumerate}[$(a)$]
	\item $bc \mid p^{ab}-1$ and $bc \nmid p^{a \ell}-1$ for all $1\le\ell\le b-1$, \msk
	
	\item  $b \mid u\,\Psi_{b}(p^a)$ and $b \nmid u\,\Psi_{\ell}(p^a)$ for all $1\le\ell\le b-1$,
	where $u=\frac{p^a-1}{c}$.
\end{enumerate}
\end{lem}

\begin{proof} 
If $bc \dagger p^{ab}-1$, then $bc$ clearly satisfies condition ($a$) in the statement.
Now, assume that $bc$ satisfies condition ($a$). We only have to prove that $bc \nmid p^{t}-1$ for all 
	$1\le t\le ab-1$.
	On the one hand, if $t<a$ then $n$ cannot divide $p^t-1$, since $c$ divides $n$ and $c$ is a primitive divisor of $p^{a}-1$. On the other hand, if $ a\le t \le m=ab$ and $n\mid p^t-1$ we necessarily have that $a\mid t$. Indeed, if $t=ad+e$ with $0\le e<a-1$ then $p^t \equiv p^e \pmod c$. 
	But $p^t \equiv 1 \pmod c$ since $c\mid n$. The primitive divisibility of $c$ implies that $e=0$, therefore $a\mid t$, that is $t=a\ell$ with $1\le \ell\le b-1$. By hypothesis, $n \nmid p^{a\ell}-1$ for all $1\le\ell \le b-1$, therefore $n\dagger p^m-1$, as desired. 
Finally, condition ($b$) is equivalent to condition ($a$) since the statements $bc\mid p^{ab}-1$ and  $bc\nmid p^{a\ell}-1$ are equivalent to $b\mid u \Psi_{b}(p^a)$ and $b\nmid u \Psi_{\ell}(p^a)$, respectively.
\end{proof}

Now, we improve a little bit the previous result by giving arithmetic conditions only on $b$ (not on $bc$), in the case $b$ is  prime, such that $bc$ is a primitive divisor of $p^{ab}-1$, provided that $c$ is a primitive divisor of $p^a-1$.
\begin{lem} \label{lem equiv n primit}
Let $p,b$ be distinct primes and let $a,c$ be positive integers such that $c\dagger p^{a}-1$. 
Then, we have that
	$$bc \dagger p^{ab}-1 \qquad \Leftrightarrow \qquad b\mid p^a -1 \quad \text{and} \quad b\nmid \tfrac{p^a-1}c .$$
In particular, $2c \dagger p^{2a}-1$ if and only if both $p$ and $\tfrac{p^a-1}{c}$ are odd (hence $c$ is even).
\end{lem}

\begin{proof}
Put $x=p^a$ and $u=\tfrac{p^a-1}{c}$. By Remark \ref{alternative expression}, if $k=\frac{p^m-1}{bc}$ then $k=\frac{u}{b}\Psi_{b}(x)$.
	
Suppose first that $bc \dagger p^{ab}-1$. Then, 
\begin{equation}\label{equiv div}
		bc \mid p^{t}-1 \qquad \Leftrightarrow \qquad b(p^a-1) \mid u(p^t-1)
\end{equation}
for all $t$. By hypothesis $n\nmid p^{t}-1$ for $1\le t<m $. 
In particular, taking $t=a$ we obtain that $b\nmid u$ by \eqref{equiv div}. 
Since $n\mid p^{m}-1$ we have that $b\mid u\,\Psi_{b}(x)$, by \eqref{equiv div}. 
Thus, $b\mid \Psi_{b}(x)$ since $b$ is prime and $b\nmid u$. By Lemma 5.2 in \cite{PV3}, 
taking $t=1$ and $h=0$, we obtain that $x\equiv 1 \pmod b$ as we wanted.
	
Now assume that $x\equiv 1\pmod{b}$ and $b\nmid u$. 
By ($b$) in Lemma \ref{red primitive}, it is enough to prove that $b\mid u \Psi_{b}(x)$ and $b\nmid u \Psi_{\ell}(x)$ for all $1\le\ell\le b-1$. 
Let $1\le \ell \le b$. Since $x\equiv 1 \pmod b$ we have that
	$$u\,\Psi_{\ell}(x) \equiv u\ell \pmod{b}.$$
Thus, since $0<\ell\le b$ and $b\nmid u$, we have that $u \,\Psi_{\ell}(x) \equiv 0 \pmod b$ if and only if $b=\ell$, as asserted. The remaining assertion is clear and the result is proved. 
\end{proof}

Thus, we obtain a reduction formula for Waring numbers in the case $b$ is prime with conditions imposed only on $b$ (and not on $bc$).

\begin{prop} \label{prod primecase}
Let $p,b$ be distinct primes and let $a,c$ be positive integers such that $c\dagger p^{a}-1$. 
If $b \mid p^a-1$ and $b\nmid \frac{p^a-1}c$, then \eqref{gk=bgu gen} holds; that is
$$g(\tfrac{p^{ab}-1}{bc},p^{ab}) = b g(\tfrac{p^a-1}c,p^a).$$
\end{prop}

\begin{proof}
By Lemma \ref{lem equiv n primit}, we have that $bc \dagger p^{ab}-1$ and hence, by Theorem~\ref{waringprod main}, expression   $g(\tfrac{p^{ab}-1}{bc},p^{ab}) = b\, g(\frac{p^a-1}c,p^a)$ holds, as desired.
\end{proof}

We will next give some instances of the previous proposition in the particular cases when $a=1$, $b=2$, or $u=\tfrac{p^a-1}c$ equals $1$ or $2$.

\begin{coro} \label{coro2}
Let $p$ and $b$ be distinct primes and let $a \in \N$. If $b \mid p^a-1$ then,  
\begin{equation} \label{gb}
	g(\tfrac{p^{ab}-1}{b(p^a-1)}, p^{ab}) = b.
\end{equation}
If, further, $p$ and $b$ are odd, then we also have
\begin{equation} \label{gb2}
	g(\tfrac{2(p^{ab}-1)}{b(p^a-1)}, p^{ab}) = 2b.
\end{equation}
\end{coro}

\begin{proof}
For \eqref{gb}, just take $c=p^a-1$ in Proposition \ref{prod primecase} and use that $g(1,q)=1$ for every $q$.
To see \eqref{gb2}, since $p$ and $b$ are odd we can take $c=\tfrac{p^a-1}{2}$ in Proposition \ref{prod primecase}. 
Winterhof proved (see ($5$) in \cite{Win}) that 
\begin{equation} \label{ineq wint}
g(k,p^a)\le k.
\end{equation}
Thus, since $g(k,q)=1$ if and only if $k=1$, we have that $g(2,p^a)=2$ for every $p$ odd, and the result follows. 
\end{proof}

\begin{rem}
Expression \eqref{gb} can be obtained from Theorems 4.1 and 6.1 in \cite{PV3}. 
\end{rem}

\begin{exam} \label{otro ejemplo}
($i$) Taking $p$ odd and $b=2$ in \eqref{gb}, for any $a\in \N$ we have 
\begin{equation} \label{fla1}
g(\tfrac{p^a+1}{2}, p^{2a})=2.
\end{equation}

\noindent ($ii$) 
Now, consider $p=2$. Since $3\mid 2^a-1$ for $a=2t$, $5\mid 2^a-1$ for $a=4t$ and $7\mid 2^a-1$ for $a=3t$ for any $t\in \N$, 
by \eqref{gb} we have
\begin{equation} \label{fla2}
\begin{aligned}
&g(\tfrac{2^{6t}-1}{3(2^{2t}-1)}, 2^{6t})   = g(\tfrac{2^{4t}+2^{2t}+1}{3}, 2^{6t}) = 3, \\ 
&g(\tfrac{2^{20t}-1}{5(2^{4t}-1)}, 2^{20t}) = g(\tfrac{2^{16t}+2^{12t}+2^{8t}+2^{4t}+1}{5}, 2^{20t}) = 5, \\
&g(\tfrac{2^{21t}-1}{7(2^{3t}-1)}, 2^{21t}) = g(\tfrac{2^{18t}+2^{15t}+2^{12t}+2^{9t}+2^{6t}+2^{3t}+1}{7}, 2^{21t}) = 7,
\end{aligned}
\end{equation}
for every $t\ge 1$. The results in \eqref{fla1} and \eqref{fla2} are part of Corollaries 6.2--6.5 in \cite{PV3}.

\noindent ($iii$)
Taking $p=7$ and $b=3$ in \eqref{gb} and \eqref{gb2}, since $7^a\equiv 1 \pmod 3$, for every $a\in \N$ we have
$$g(\tfrac 13 (7^{2a}+7^a+1), 7^{3a})=3 \qquad \text{and} \qquad g(\tfrac 23 (7^{2a}+7^a+1), 7^{3a})=6.$$ 
For instance,  
$g(19, 343) = 3$, $g(817, 117{.}649)=3$, $g(39{.}331, 40{.}353{.}607)=3$ and $g(38, 343)=6$, 
$g(1{.}634, 117{.}649) =6$ and $g(78{.}662, 40{.}353{.}607) =6$,
for $a=1$, $2$ and $3$ respectively. 
\hfill $\lozenge$
\end{exam}

\begin{coro} \label{coro1}
Let $p$ be an odd prime and $a,c$ integers such that $c\dagger p^a-1$ (hence $c$ is even).
If $\tfrac{p^a-1}{c}$ is odd, then 
\begin{equation} \label{g2b}
	g(\tfrac{p^{2a}-1}{2c}, p^{2a})=2 g(\tfrac{p^a-1}{c},p^a).
\end{equation}
\end{coro}

\begin{proof}
Just take $b=2$ in Proposition \ref{prod primecase}.
\end{proof}

\begin{exam} \label{g=4}
For any odd prime $p>3$ we have 
\begin{equation}\label{flap4}
\begin{aligned}
	& g(\tfrac 32 (p^2+1), p^4)=4, \\[1mm]
	& g(4(p^2+1), p^4)=4, \quad \qquad \text{for $p\equiv \pm 3 \, ($mod $8), \, p\ge 67$}.
\end{aligned}
\end{equation} 
For instance, for the smallest cases $p=5$ and $p=67$, respectively, we have
	$$g(39,625)=4 \qquad \text{and} \qquad g(6{.}735, 20{.}151{.}121) = g(17{.}960, 20{.}151{.}121)=4.$$

To prove \eqref{flap4} we use Corollary \ref{coro1} with $a=2$ together with the following result 
\begin{equation} \label{small g2}
k\mid q-1 \quad \text{and} \quad 2\le k\le \sqrt[4]{q} +1 \qquad \Rightarrow \qquad g(k,q)=2
\end{equation}
due to Small (\cite{Sm}).
 For the first expression, take $c=\tfrac{p^2-1}{3}$; hence $c\dagger p^2-1$ and by \eqref{g2b}
we have 
	$$g(\tfrac 32 (p^2+1), p^4) = 2g(3,p^2)=4,$$
since $g(3,p^2)=2$ by \eqref{small g2}.  
For the second expression, take $p\equiv \pm 3 \pmod 8$ and $c=\tfrac{p^2-1}{8}$. 
Thus, $c\dagger p^2-1$ and by \eqref{g2b}
we have 
	$$g(4(p^2+1), p^4) = 2g(8,p^2)=4,$$
since $g(8,p^2)=2$ for any $p^2 \ge 8$, by \eqref{small g2} again. \hfill $\lozenge$
\end{exam}

\begin{coro} \label{coro3}
Let $p$ and $b$ be distinct primes and $c \in \N$ such that $b,c \mid p-1$. If $b \nmid \tfrac{p-1}{c}$, then 
\begin{equation} \label{ga1}
	g(\tfrac{p^{b}-1}{bc}, p^{b}) =b g(\tfrac{p-1}{c},p).
\end{equation}
\end{coro}

\begin{proof}
Just take $a=1$ in Proposition \ref{prod primecase}.
\end{proof}

As a particular interesting cases of the above corollary we have the following. 
\begin{coro} \label{coro4}
Let $b$ be a prime and $p$ a prime of the form $p=bt+1$, $t\in \N$, with $b\nmid t$. Then, 
\begin{equation} \label{gbc}
	g(\tfrac{p^{b}-1}{b^2}, p^{b})=b g(t,p).
\end{equation}
 In particular, for all primes of the form $p=2b+1$ with $b$ odd we have 
\begin{equation} \label{gbc2}
	g(\tfrac{p^{b}-1}{b^2}, p^{b})= 2b.
\end{equation}
\end{coro}

\begin{proof}
For the first expression just take $b=c$ in Corollary \ref{coro3}. The second one follows by Small's result \eqref{small g2}, since $g(2,p)=2$ for every odd prime $p$. 
\end{proof}

\begin{coro} \label{coro5}
Let $p$ be an odd prime. If $p\equiv 3\pmod{4}$ then, 
\begin{equation} \label{gb p-1}
	g(\tfrac{p^{2}-1}{4}, p^{2})=p-1.
\end{equation}
If $p\equiv 1 \pmod 4$ then $g(\tfrac{p^{2}-1}{4}, p^{2})$ does not exist.
\end{coro}

\begin{proof}
Just take $b=c=2$ in Corollary \ref{coro4}.
Hence $p=2t+1$ with $t$ odd, that is $p\equiv 3 \pmod 4$ and we have that
$g(\tfrac{p^{2}-1}{4}, p^{2})=2 g(\tfrac{p-1}{2},p)=p-1$,
since it is well-known that $g(\frac{p-1}2,p) = \frac{p-1}2$ for every prime $p$.
The remaining assertion follows by \eqref{existence}.
\end{proof}

\begin{exam}
($i$) Note that the primes $p=7, 11, 23, 47$ satisfy $p=2b+1$ with prime $b=3, 5, 11, 23$ respectively. Thus, by Corollary \ref{coro4}, we have that
$g(38,343) = g(\tfrac{7^3-1}{9}, 7^3) = 6$, $g(6{.}442,161{.}051)= g(\tfrac{11^5-1}{25}, 11^5) = 10$ and 
$$g(7{.}874{.}460{.}809{.}206, 952{.}809{.}757{.}913{.}927) = g(\tfrac{23^{11}-1}{121}, 23^{11}) =22.$$ 
The next Waring number is $g(\tfrac{47^{23}-1}{529}, 47^{23}) = 46$, 
where $k=\tfrac{47^{23}-1}{529}$ equals 
{\footnotesize $$542{.}994{.}037{.}206{.}173{.}139{.}403{.}790{.}190{.}513{.}425{.}518, 287{.}243{.}845{.}682{.}065{.}590{.}744{.}605{.}010{.}781{.}602{.}099{.}023.$$}

\noindent ($ii$)
Taking $p=7,11, 19, 23, 31, 43$ in Corollary \ref{coro5}, by \eqref{gb p-1} we have that   
$g(12, 49)=6$, $g(30, 121)=10$, $g(90, 361)=18$, $g(132, 529)=22$, $g(240, 961)=30$, $g(462, 1{.}849)=42$.
 \hfill $\lozenge$
\end{exam}

\section{Arithmetic properties of $\Psi_{b}(x)$}
As we have already mentioned in the previous section (see Remark \ref{alternative expression} and Lemma~\ref{red primitive}), the Waring numbers $g(k,q)$ that we are considering can be put in terms of the number 
$$\Psi_b(p^a) = \tfrac{p^{ab}-1}{p^a-1} = p^{a(b-1)}+ \cdots + p^{2a}+ p^a+1,$$ 
that is $k=\tfrac 1b \Psi(p^a)$. Thus, we need that $b$ divides $\Psi_b(p^a)$. The case when $b$ is prime was considered in Section 3.
In this section we give some divisibility properties for $\Psi_b(x)$, with $x\in \Z$, that will be used in the next section 
to obtain reduction formulas for Waring numbers $g(k,p^{ab}) $ in some general cases.
The section is somehow technical and can be omitted at first reading. 

We denote by $ord_b(a)$ the order of $a$ modulo $b$.
We point out that 
\begin{equation} \label{equiv ord}
n\dagger p^m-1 \qquad \Leftrightarrow \qquad ord_{n}(p)=m.
\end{equation} 
We begin by recalling Lemmas 5.2 and 5.3 from \cite{PV3}, that we state here together.

\begin{lem}[\cite{PV3}] \label{lema 1}
Let $r, r_1, \ldots, r_s$ be prime numbers and $x,t, t_1, \ldots,t_s$ be positive integers for $s\ge 1$. Then we have the following:
\begin{enumerate}[$(a)$]
	\item If $(x,r)=1$, then $r^t\mid \Psi_{r^t}(x)$ if and only if $ord_{r^t}(x)=r^h$ for some $0\le h \le t-1$. 
	\sk 
	
	\item If $x$ is coprime with $b = r_1^{t_1} \cdots r_s^{t_s}$ and $ord_{r_{i}^{t_i}}(x)=r_{i}^{h_i}$ with $0\le h_i \le t_{i}-1$ for all $i=1,\ldots,s$ then $b \mid \Psi_{b}(x)$.
\end{enumerate}
\end{lem}

We now give some divisibility properties for $\Psi_b(x)$ in the case when $b$ is a prime power. 
For simplicity, in the rest of the section for any $t\in \N$ we will use the notation 
\begin{equation} \label{notation}
I_{t}'=\{1,2,\ldots, t-1\}.
\end{equation}

\begin{prop} \label{Psi pow cond}
Let $r$ be an odd prime, $t, h, \ell \in \N$ and let  
$\beta \in (\mathbb{Z}_{r^t})^*$ with $ord_{r^t}(\beta)=r^{h}$. Then, we have the following:
\begin{enumerate}[$(a)$]
	\item If $2h\le t$ then 
$\Psi_{r^h}(\beta)\equiv r^h \pmod{r^t}$. \msk 
	
	\item If $2h\le t$ and $r^h\nmid \ell$ then 
		$r^h\nmid \Psi_{\ell}(\beta)$. \msk 
		
	\item For any $e \in \N$ and $\ell \in I_{r^h}'$ we have
	\begin{equation} \label{eq p2}
		\Psi_{er^h+\ell}(\beta) \equiv e \Psi_{r^h}(\beta) + \Psi_{\ell}(\beta)\pmod{r^t}.
	\end{equation}
			
	\item For any $a,b \in \mathbb{N}$ we have $\Psi_{ab}(\beta)=\Psi_{b}(\beta^a) \Psi_{a}(\beta)$.
\end{enumerate}	
\end{prop}

\begin{proof}
To prove $(a)$, notice that
	$$\Psi_{r^h}(\beta)=1+\sum_{i=1}^{r^h-1}\beta^i=1+\tfrac 12 \sum_{i=1}^{r^h-1}(\beta^i+\beta^{-i}).$$
Now, taking into account that $\beta^i+\beta^{-i}=\tfrac{(\beta^i-1)^2}{\beta^i}+2$, we have 
	$$\Psi_{r^h}(\beta) = 1 + 2(\tfrac{r^h-1}{2}) + \tfrac 12 \sum_{i=1}^{r^h-1}\tfrac{(\beta^i-1)^2}{\beta^i}=r^h+\tfrac 12 \sum_{i=1}^{r^h-1}\tfrac{(\beta^i-1)^2}{\beta^i}.$$
Since $r$ is odd, it is enough to show that 
	$$\sum_{i=1}^{r^h-1} \tfrac{(\beta^i-1)^2}{\beta^{i}} \equiv 0 \pmod{ r^t}.$$ 
By basic theory of $\mathbb{Z}^*_{r^u}$, we can choose some $\alpha$ which generates $\Z^*_{r^u}$ 
for all $u\ge 1$ such that 		
	$$\beta=\alpha^{j(r-1)r^{t-1-h}}$$ 
with $j$ coprime with $r$. Clearly $(r-1)r^{t-1-h}=\varphi(r^{t-h})$ and by Euler's Theorem 
	$$\beta\equiv \alpha^{j\varphi(r^{t-h})}\equiv 1 \pmod{r^{t-h}}.$$ 
Thus $r^{t-h}\mid \beta-1$, which implies that for all $i$ we have 
\begin{equation}\label{division}
	r^{t-h}\mid \beta^{i}-1.
\end{equation}
Then, we have that $r^{2t-2h}\mid (\beta^i-1)^2$. 
Notice that $t\le 2t-2h$ since by hypothesis $2h\le t$. This implies that 
$r^t \mid r^{2t-2h}$ and hence $r^t \mid (\beta^i-1)^2$. Therefore, we obtain
$\sum_{i=1}^{r^h-1} (\beta^i-1)^2) \beta^{-i} \equiv 0 \pmod {r^t}$, as desired.
	
To see $(b)$, notice that $2h\le t$ implies $h\le t-h$ and then $r^h\mid r^{t-h}$. On the other hand, by \eqref{division} we have $r^{t-h}\mid \beta^i-1$ for all $i$. Hence $\beta^{i} \equiv 1 \pmod{r^h}$ for all $i$ and we obtain  
	$$\Psi_{\ell}(\beta)\equiv \ell \pmod{r^h}.$$
By hypothesis, we have $r^h\nmid \ell$ and therefore $r^h\nmid \Psi_{\ell}(\beta)$, as we wanted.
	
For $(c)$, we have that
	$$\Psi_{er^h+\ell}(\beta)\equiv \sum_{i=0}^{er^h-1}\beta^{i} + \sum_{i=er^h}^{er^h+\ell-1}\beta^i\equiv 
	\sum_{j=1}^{e} \sum_{i=(j-1)r^h}^{jr^h-1}\beta^{i} + \sum_{i=er^h}^{er^h+\ell-1} \beta^i \pmod{r^t} .$$
By cyclicity, since $ord_{r^t}(\beta)=r^h$, we obtain 
	$$\sum_{i=(j-1)r^h}^{jr^h-1}\beta^{i}\equiv \sum_{i=0}^{r^{h}-1}\beta^i\equiv \Psi_{r^h}(\beta)\pmod{r^t}$$ 
and 
	$$\sum_{i=er^h}^{er^h+\ell-1}\beta^i\equiv \Psi_{\ell}(\beta)\pmod{r^t}.$$
Clearly, equation \eqref{eq p2} is a consequence of these last two congruence equalities. 
	
Finally, $(d)$ follows from the identity $\Psi_{\ell}(\beta)=\frac{\beta^{\ell}-1}{\beta-1}$ for $\beta\ne 1$, and the proposition is proved.
\end{proof}

The following result will be crucial in the proof of the reduction formula for Waring numbers that we will give in the next section.

\begin{prop} \label{primitive pow}
	Let $r$ be an odd prime and let $x, t, h\in \N$. 
	If $ord_{r^t}(x)=r^h$ for some $0\le h\le t-1$, then $r^{t} \nmid \Psi_s(x)$ for all $s \in \{1,\ldots, r^t-1\}$.
\end{prop}

\begin{proof}
For convenience, we first prove three claims. 	
	
\smallskip 	
\noindent
\textit{Claim 1}. The assertion in the statement is true for $2h\le t$.
	
Suppose that $2h\le t$ and let $s\in I_{r^t}'$. 
On the one hand, if $r^h\nmid s$, by part $(b)$ of Lemma \ref{Psi pow cond} we have that $r^h\nmid \Psi_{s}(x)$ and then $r^{t}\nmid \Psi_{s}(x)$. On the other hand, if $s=\ell r^{h}$ for some $\ell$, by parts $(a)$ and $(c)$ of Lemma \ref{Psi pow cond} we have
	$$\Psi_{s}(x)=\Psi_{\ell r^{h}}(x)\equiv \ell\Psi_{r^h}(x)\equiv \ell r^{h} \equiv s \pmod{r^t}.$$  
	Therefore $r^t\nmid \Psi_{s}(x)$ for all $s=1,2,\ldots,r^{t}-1$.
	\hfill $\lozenge$
	
\smallskip
\noindent
\textit{Claim 2}. If $ord_{r^t}(x)=r^h$ for some $h\le t-1$, then $r^{h+1}\nmid \Psi_{r^h}(x)$.
	
It is not difficult to see that if $2h+1\le t$, then $r^{h+1}\nmid \Psi_{\ell}(x)$ for all $\ell \in I_{r^{h+1}}'$,
since in this case $ord_{r^{h+1}}(x)=1$. In particular, $r^{h+1}\nmid \Psi_{r^h}(x)$. 
	
Suppose that $r^{h+1}\mid \Psi_{r^h}(x)$ for some $h$ such that $ord_{r^t}(x)=r^h$. Thus, there exists a minimal integer $h$ with $\frac {t-1}2<h< t$ satisfying both $r^{h+1}\mid \Psi_{r^h}(x)$ and $ord_{r^t}(x)=r^h$.
	
By part $(d)$ of Lemma \ref{Psi pow cond}, we have $\Psi_{r^h}(x)=\Psi_{r^{h-1}}(x^r)\Psi_{r}(x)$ and hence $r^h\mid \Psi_{r^{h-1}}(x^r)$ or $r^2\mid \Psi_{r}(x)$, since $r\mid \Psi_{r}(x)$ because $x\equiv 1 \pmod{r}$.
	
Notice that if $ord_{r^t}(x)=r^h$,  then $ord_{r^t}(x^r)=r^{h-1}$, by minimality, we obtain that $r^h \nmid \Psi_{r^{h-1}}(x^r)$. 
On the other hand, since the order of $x$ is a power of $r$ modulo $r^2$, we have that $ord_{r^2}(x)$ is $r$ or $1$. Clearly, 
if $ord_{r^2}(x)=1$ then $\Psi_{r}(x)\equiv r \pmod{r^2}$ and thus $r^2\nmid \Psi_{r}(x)$. 
If, otherwise, $ord_{r^2}(x)=r$, by Claim~1 we obtain that $r^2\nmid \Psi_{r}(x)$. These facts imply that $r^{h+1} \nmid \Psi_{r^{h}}(x)$. \hfill $\lozenge$

\smallskip
\noindent
\textit{Claim 3}. If $r^h\nmid \Psi_{\ell}(x)$ for all $\ell \in I_{r^h}'$ and $r^{h+1}\nmid \Psi_{r^h}(x)$ then $r^t\nmid \Psi_{s}(x)$ for all $s\in I_{r^t}'$. 
	
Let $s\in I_{r^t}'$ such that $r^{h}\nmid s$. Then, $s=er^{h}+\ell$ for some $\ell\in I_{r^h}'$. 
By part $(c)$ of Lemma~\ref{Psi pow cond}, we have that  
	$$\Psi_{s}(x)\equiv e\Psi_{r^h}(x)+\Psi_{\ell}(x)\pmod{r^t}.$$
On the first hand we have that  $r^h\mid r^t$ since $h<t$, and thus the previous congruence holds modulo $r^h$. 
On the other hand, we have $r^h\mid \Psi_{r^h}(x)$ since $ord_{r^h}(x)=1$ or $r^{2h-t}$ depending whether $2h\le t$ or $2h>t$, respectively. Thus, we have that 
	$$\Psi_{s}(x)\equiv \Psi_{\ell}(x)\pmod{r^t}.$$
By hypothesis, $r^h\nmid \Psi_{\ell}(x)$ for all $\ell\in I_{r^h}'$.  
This implies that $r^h\nmid \Psi_{s}(x)$ for all $s\in I_{r^t}'$ with $r^h\nmid s$. By transitivity,
$r^t\nmid \Psi_{s}(x)$ for all $s\in I_{r^t}'$ 
such that $r^h\nmid s$.
	
Assume now that $s=er^h$ with $e \in I_{r^{t-h}}'$. 
By part ($c$) of Lemma~\ref{Psi pow cond}, we have that 
	$$\Psi_{s}(x)\equiv e \Psi_{r^{h}} \pmod{r^t}.$$
Notice that if $r^t \mid \Psi_{s}(x)$, necessarily $r^{h+1}\mid \Psi_{r^h}(x)$ since the $r$-adic value of $e$ is strictly less than $r^{t-h}$. By hypothesis $r^{h+1}\nmid \Psi_{r^h}(x)$ and then $r^t \nmid  \Psi_{s}(x)$  for $s=er^h$ with $e\in I_{r^{t-h}}'$. This complete all the cases, i.e.\@  $r^t\nmid \Psi_{s}(x)$ for all $s\in I_{r^{t-h}}'$ and the claim is proved. 
\hfill $\lozenge$

Now, we proceed by induction on $t$. Claim 1 implies that the statement holds for $t=1$ and $t=2$.
Suppose now that $t>1$ and the statement holds for all of $t'< t$. By Claim~1, it is enough to show the assertion for $2h>t$. 
Then, assume that $2h>t$. 
By taking $h'=2h-t$, we obtain that 
	$$ord_{r^h}(x)=r^{h'},$$ 
since $h<t$. By induction, we obtain that $r^h\nmid \Psi_{\ell}(x)$ for all $\ell\in I_{r^h}'$ and by Claim 2 we have that 
$r^{h+1} \nmid \Psi_{r^h}(x)$. Therefore, Claim 3 implies that 
$r^{t}\nmid \Psi_{s}(x)$ for all $s\in I_{r^t}'$, as desired.
\end{proof}

\section{Reduction formula for arbitrary $b$}
In this section, using the results from the previous one, we give conditions to have a reduction formula for Waring numbers of the form $g(\tfrac{p^{ab}-1)}{b(p^a-1)}, p^{ab}) = b  g(\tfrac{p^a-1}{c}, p^a)$
in some different cases, for arbitrary integers $b$ not a power of $2$.

We first consider the case when $b$ is a prime power. We will need a preliminary result for powers of 2 in the vein of Lemma \ref{lem equiv n primit}.

\begin{lem}\label{lem pow 2}
Let $p$ be an odd prime and suppose that $a$ and $c$ are positive integers such that $c\dagger p^{a}-1$ and $b=2^t$ for some integer $t\ge 2$. 
Then, we have that
$$bc \dagger p^{ab}-1 \qquad \Leftrightarrow \qquad p^{a}\equiv 1\pmod{4} \quad \text{and} \quad \tfrac{p^a-1}c \text{ odd} .$$
In this case, $c$ is even with either $p\equiv 1 \pmod 4$ or else $p\equiv 3 \pmod 4$ and $a$ is even.
\end{lem}
Note that the equivalent condition for $bc \dagger p^{ab}-1$ does not actually depend on $b$.

\begin{proof}
($\Rightarrow$)	Assume first that $bc \dagger p^{ab}-1$. Thus $bc \nmid p^{a\ell}-1$ for $\ell\in\{0,\ldots,b-1\}$, 
	in particular $bc\nmid p^{2^{t-1}a}-1$. We will use the following relation
	$$p^{2^{t-1}a}-1=(p^a-1)\prod_{0\le i \le t-2}(p^{2^{i}a}+1).$$
	Since $b=2^t$ and $c\mid p^{a}-1$, the condition $bc\nmid p^{2^{t-1}a}-1$ is equivalent to
\begin{equation} \label{v2a}	
v_{2}(bc)> v_{2}(p^{2^{t-1}a}-1),
\end{equation}
		where $v_2$ denotes the $2$-adic valuation, which in turn is equivalent to
\begin{equation*} \label{v2b}
t > v_{2}(\tfrac{p^a-1}{c}) +\sum_{0\le i\le t-2} v_{2}(p^{2^{i}a}+1).
\end{equation*}

If $t=2$, then $v_{2}(\tfrac{p^a-1}{c})=0$ and $v_{2}(p^{a}+1)=1$, otherwise \eqref{v2a} would not hold,
since $v_{2}(\tfrac{p^a-1}{c})\ge 0$ and $v_{2}(p^a+1)\ge 1$.
Hence, $\frac{p^{a}-1}{c}$ is odd and $p^{a}+1\equiv 2 \pmod {4}$, that is $p^a\equiv 1 \pmod{4}$.
	
Now, assume $t>2$ and notice that $p^{2^{i}a}\equiv 1\pmod{4}$ for all $i\in I_{t-1}'$, in the notation of \eqref{notation}. This implies that $v_{2}(p^{2^{i}a}+1)=1$ for all $i\in I_{t-1}'$. 
Thus, we obtain that $t> v_{2}(\tfrac{p^a-1}{c}) + v_{2}(p^{a}+1)+t-2$, which is equivalent to
	$$2> v_{2}(\tfrac{p^a-1}{c}) + v_{2}(p^{a}+1).$$
By using the same argument as in the case $t=2$, we obtain that $\frac{p^a-1}{c}$ is odd and $p^{a}\equiv 1\pmod{4}$, as desired.
	
\sk 
\noindent ($\Leftarrow$)
Now assume that $p^a\equiv 1\pmod{4}$ and $\frac{p^a-1}{c}$ is odd. By Lemma \ref{red primitive}, it is enough to show that
	$$bc\mid p^{ab}-1 \qquad \quad \text{and} \quad \qquad bc\nmid p^{a\ell}-1 \:\: \text{ for } \:\: \ell \in I_b'.$$ 
First, notice that if $\ell$ is odd, then $\frac{p^{a\ell}-1}{c}$ is also odd, 
since $v_{2}(p^{a\ell}-1)=v_{2}(p^{a}-1)=v_2(c)$ and therefore $b\nmid \frac{p^{a\ell}-1}{c}$ for $\ell$ odd. 
Now, assume that $\ell=2^{h}v$ with $v$ odd and $h\ge 1$. Thus,
$h\le t-1$ since $\ell\le b-1$. We have the following factorization 
	$$p^{a\ell}-1= (p^{av}-1)\prod_{0\le i \le h-1}(p^{2^{i}av}+1).$$
The hypothesis $p^{a}\equiv 1\pmod{4}$ implies that $v_{2}(p^{2^{i}av}+1)=1$ for all $i\in\{0,\ldots,h-1\}$ 
and so we obtain that
	$$v_{2}\Big(\prod_{0\le i \le h-1} (p^{2^{i}av}+1) \Big) = \sum_{0\le i \le h-1} v_{2}(p^{2^{i}av}+1)=h.$$ 
Hence $v_{2}(\frac{p^{a\ell}-1}{c})=h$, since $v_{2}(\frac{p^{av}-1}{c})=0$ ($c$ divides $p^{av}-1$ since $p^{a}-1$ is a divisor of $p^{av}-1$).
Thus, we have that $v_{2}(b)=t > h = v_{2}(\frac{p^{a\ell}-1}{c})$, which implies that $b\nmid \frac{p^{a\ell}-1}{c}$ and then $bc\nmid p^{a\ell}-1$. 
A similar argument shows that $bc\mid p^{ab}-1$. 
Therefore  $bc\dagger p^{ab}-1$ as asserted.	
The final remark is straightforward and the proof is complete. 
\end{proof}

We are now in a position to give the reduction formula for certain Waring numbers $g(k_b, p^{ab})$ in the case when $b$ is a prime power. 
\begin{prop} \label{red pow}
Let $p$ be a prime and suppose that $a,b,c \in \N$ are such that $c\dagger p^{a}-1$, $b=r^t$ for some integer $t\ge 2$ 
with $r$ a prime different from $p$, and $(\frac{p^a-1}{c},r)=1$. 
Then, we have
\begin{equation} \label{fla rt}
g( \tfrac{p^{ar^t}-1}{cr^t}, p^{ar^t} ) = r^{t}  g(\tfrac{p^a-1}c, p^a)
\end{equation}
either if $r=2$ and $p^{a}\equiv 1\pmod{4}$ or else if $r$ is odd and $ord_{b}(p^a)=r^h$ for some $0\le h\le t-1$.
\end{prop}

\noindent \textit{Note.} The case $r=2$ in the proposition generalizes Corollary \ref{coro1}.

\begin{proof}
By Theorem \ref{waringprod main}, it is enough to prove that in both cases we have $bc \dagger p^{ab}-1$ and hence \eqref{gk=bgu gen} holds, which takes the form \eqref{fla rt}. If $r=2$ and $p^a\equiv 1 \pmod 4$, the result follows directly by Lemma \ref{lem pow 2}, since $(\frac{p^a-1}{c},r)=1$ implies that $\frac{p^a-1}{c}$ is odd.
	
Now, assume that $r$ is odd and $ord_{b}(p^a)=r^h$ with $0\le h\le t-1$. 
Let us consider $m=ab=ar^t$ and $n=bc$, by Lemma \ref{red primitive}, 
it is enough to prove that $n$ divides $p^m-1$ and $n$ does not divide $p^{a\ell}-1$ for all $\ell \in I_{r^t}'$. 
Let $x=p^a$. By hypothesis $(r,\frac{p^a-1}{c})=1$, then these two conditions are equivalent to $r^t\mid \Psi_{r^t}(x)$ and $r^t\nmid \Psi_{\ell}(x)$ for all $\ell \in I_{r^t}'$, respectively (see item ($b$) in Lemma~\ref{red primitive}).
We have $ord_{r^t}(x)=r^h$, by hypothesis, and hence $r^t\mid \Psi_{r^t}(x)$, by part ($a$) in Lemma~\ref{lema 1}. That is, $n\mid p^m-1$. On the other hand, it follows from Proposition~\ref{primitive pow} that $r^t\nmid \Psi_{\ell}(x)$ for all $1\le\ell\le r^{t}-1$. Therefore, $n$ is a primitive divisor of $p^m-1$, as desired.
\end{proof}

We now give the reduction formula \eqref{gk=bgu gen} in its most generality, for any integer $b$ not a power of $2$.
\begin{thm} \label{teo 42}
Let $p$ be a prime and let $a,b,c$ be positive integers such that $c \dagger p^a-1$ and
$b = 2^t r_1^{t_1} \cdots r_s^{t_s}$ with $r_1, \ldots, r_s$ distinct odd primes, $t\ge 0$, $t_1, \ldots, t_s\ge 1$ and $(p,b)=1$. 
If $(\frac{p^a-1}{c},b)=1$ and $ord_{r_{i}^{t_i}}(p^a) = r_i^{h_i}$ with $0\le h_i\le t_i-1$ for all $i=1,\ldots,s$, then the reduction formula \eqref{gk=bgu gen} holds for $t=0,1$ and for any $t\ge 2$ if further $p^{a}\equiv 1 \pmod 4$.
That is, under the above conditions, if $b'=  r_1^{t_1} \cdots r_s^{t_s}$ we have 
\begin{equation} \label{red fla gral even pow}
g( \tfrac{p^{2^{t}ab'}-1}{2^{t}b'c},p^{2^tab'} ) = 2^{t} b'  g(\tfrac{p^{a}-1}{c},p^a).
\end{equation}
\end{thm}

\begin{proof}
Assume first that $t=0$ (hence $b=b'$ is odd). In this case we  prove \eqref{red fla gral even pow} by induction on $s$. 
The case $s=1$ is a direct consequence of Proposition \ref{red pow}. 
So, assume that $s>1$. 
Recall that if we put $u=\frac{p^a-1}{c}$, then 
$\tfrac{p^{ab}-1}{bc}= \tfrac{u}{b}\Psi_{b}(p^a)$, by Remark \ref{alternative expression}.
By $(d)$ of Lemma~\ref{Psi pow cond}, if we put $b_{s}=\frac{b}{r_{s}^{t_s}}$ we have that 
	$$\Psi_{b}(p^a) = \Psi_{b_{s}r_{s}^{t_s}}(p^a) = \Psi_{r_{s}^{t_s}}(p^{ab_s}) \Psi_{b_s}(p^a).$$ 
Let $u'=\frac{u}{b_s}\Psi_{b_s}(p^a)$. Note that $u' \in \Z$, 
since $b_s\mid \Psi_{b_s}(p^a)$ by ($b$) in 
Lemma~\ref{lema 1}, thus we have
	$$g(\tfrac{u}{b}\Psi_{b}(p^a),p^{ab})=g(\tfrac{u'}{r_{s}^{t_s}}\Psi_{r_{s}^{t_s}}(p^{ab_s}), p^{ab_s r_{s}^{t_s}}) .$$ 
Since $ord_{r_{s}^{t_s}}(p^a)$ is a power of $r_s$, then the order of $p^{ab_s}$ is a power of $r_s$ as well. 
	
Now we prove that $u'$ and $r_{s}$ are coprime. Since the order of $p^a$ modulo $r_s^{t_s}$ is a power of $r_s$, then the order of $p^{a}$ modulo $r_s$ is $1$, i.e\@ $p^{a}\equiv 1\pmod{r_s}$. Thus,  
$\Psi_{b_s}(p^a)\equiv b_s \pmod{r_s}$ 
and hence $r_s\nmid \Psi_{b_s}(p^a)$ since $(b_s,r_s)=1$. On the other hand, by hypothesis, $r_s$ is coprime with $u$. 
Therefore $r_s$ is coprime with $u'$ since $u'$ is a divisor of $u \Psi_{b_s}(p^a)$. 
	
By Proposition \ref{red pow} we obtain   
	$$g(\tfrac{u'}{r_{s}^{t_s}}\Psi_{r_{s}^{t_s}}(p^{ab_s}),p^{ab_s r_{s}^{t_s}}) = r_{s}^{t_s}  g(u',p^{ab_s}).$$
By the inductive hypothesis we have $g(u',p^{ab_s})=b_s g(u,p^a)$, and hence 
	$$g(\tfrac{p^{ab}-1}{bc},p^{ab})=g(\tfrac{u}{b}\Psi_{b}(p^a),p^{ab}) = r_{s}^{t_s}b_s  g(u,p^a) = b  g(u,p^a),$$ 
as we wanted to see. 
	
Suppose now that $t=1$. Since $b'=\frac b2$, by Corollary \ref{coro1} it is enough to show that $\frac{p^{ab'}-1}{b'c}$ is odd or, equivalently, that $v_{2}(b'c)= v_{2}(p^{ab'}-1)$. 
By hypothesis $(\frac{p^{a}-1}{c},b)=1$ and so $\frac{p^{a}-1}{c}$ is odd, that is $v_{2}(c)=v_{2}(p^{a}-1)$. 
Let $\ell=v_{2}(c)$, then $\ell =v_{2}(b'c)$ since $b'$ is odd. On the other hand, since
$p^{a}\equiv 1\pmod{2^{\ell}} $ then $p^{ab'}\equiv 1\pmod{2^{\ell}}$. It is enough to show that
$2^{\ell+1}\nmid p^{ab'}-1$. Since $\mathbb{Z}_{2^{\ell+1}}^*$ has order $2^{\ell}$, then the order of all its elements is a power of $2$. 
By taking into account that $p^{a}\in \mathbb{Z}_{2^{\ell+1}}^*$ and $p^{a}\not \equiv 1 \pmod{2^{\ell+1}}$, 
we have that $p^{ab'}\not \equiv 1 \pmod{2^{\ell+1}}$, otherwise $b'$ is divisible by $2$ which cannot happen because $b'$ is odd.
Hence $v_{2}(p^{ab'}-1)=\ell=v_{2}(b'c)$, that is $\frac{p^{ab'}-1}{b'c}$ is odd, as desired. 
	
Finally, take $t\ge 2$. By Proposition \ref{red pow}, it is enough to show that $p^{ab'}\equiv1 \pmod 4$ and $\frac{p^{ab'}-1}{b'c}$ is odd where $b'=\frac{b}{2^t}$. 
The same argument as in the above paragraph implies that $\frac{p^{ab'}-1}{b'c}$ is odd, since $\frac{p^{a}-1}{c}$ is odd by hypothesis.
On the other hand, clearly the hypothesis $p^{a}\equiv 1\pmod{4}$ implies that $p^{ab'}\equiv 1\pmod{4}$.
Expression \eqref{red fla gral even pow} is now a direct consequence of the item ($a$) of Proposition \ref{red pow} and the case $t=0$.
\end{proof}

Recall that if $b=r_1^{t_i} \cdots r_s^{t_s}$ is the prime decomposition of $b$ then the radical of $b$ is $rad \, (b)=r_1 \cdots r_s$. As a direct consequence of the theorem we have the following.

\begin{coro} \label{fr rad b}
Let $p$ be a prime and let $a,b,c \in \N$ such that $c \dagger p^a-1$, $p\nmid b$ and $b>1$ is not a power of $2$.
Thus, we have: 
\begin{enumerate}[$(a)$]
	\item If $\varphi(rad \, (b)) \mid  a$ and $(\frac{p^a-1}{c},b)=1$, then formula \eqref{gk=bgu gen} holds. \msk
	
	\item If $(\frac{p^{\varphi(rad \, (b))}-1}{c},b)=1$, then one can also replace $a$ by $\varphi(rad \, (b))$ in \eqref{gk=bgu gen}, that is
	\begin{equation} \label{g coro}
	g \big( \tfrac{p^{b \varphi(rad \, (b))}-1}{bc}, p^{b \varphi(rad \, (b))} \big) = 
	b g(\tfrac{p^{\varphi(rad\, (b))}-1}{c},p^{\varphi(rad\, (b))}).
	\end{equation}
\end{enumerate}
\end{coro}

\begin{proof}
Let $x=p^a$. 
Notice that since $b>1$ is not a power of $2$, we have that $\varphi(rad(b))$ is even. 
Thus $a$ is even and so we have that $x\equiv 1\pmod 4$.
By Theorem \ref{teo 42} it is enough to show that $ord_{r_{i}^{t_i}}(x)=r_{i}^{h_i}$ for some $0\le h_i\le t_{i}-1$ for every $i$, where $r_i$ are the odd prime factors of $b$ with $v_{r_i}(b)=t_i$. 
Note that 
$\varphi(r_{i}^{t_i})=r_{i}^{t_i-1}(r_i-1)$ and 
$$\varphi(rad(b)) = \varphi(r_1 \cdots r_\ell) = (r_1-1)\cdots (r_\ell-1).$$ 
Since $r_i-1 \mid a$ for $i=1,\ldots,s$, because $r_i-1\mid\varphi(rad(b))$ and $\varphi(rad(b))\mid a$ by hypothesis, the Euler-Fermat's theorem implies that 
$$x^{r_{i}^{t_i-1}} = p^{ar_{i}^{t_i-1}} = (p^{\varphi(r_i^{t_i})})^{\frac{a}{r_i-1}}\equiv 1 \pmod {r_{i}^{t_i}}$$ 
for all $i=1,\ldots,s$.
This implies that $ord_{r_i^{t_i}}(x)\mid r_i^{t_i}$ and then, for each $i=1,\ldots,s$ there is some $0\le h_i \le t_i -1$ such that $ord_{r_i^{t_i}}(x)=r_i^{h_i}$, thus proving ($a$).

The statement in ($b$) follows directly by taking $a=\varphi(rad\, (b))$ in case ($a$).
\end{proof}
\black

\begin{exam}
Let $b=3^{2} 5^{2}=225$, hence $a=\varphi(rad(b))=\varphi(15)=8$, and take $c=p^a-1$.
Thus, by the previous corollary, we have that 
$$g(\tfrac{2(p^{1800}-1)}{225(p^{8}-1)}, p^{1800}) = 225 g(2,p^{1800}) = 450$$ 
for all prime $p>5$, since $g(2,q)=2$ for all prime power $q$. \hfill $\lozenge$
\end{exam}

\subsection*{Kononen's result as a particular case} 
There are many upper bounds for general Waring numbers, but very few exact formulas (see \S2 in \cite{PV3} for a brief survey).
From these exact formulas, the one given by Kononen in 2010 stands up, generalizing a previous result of Winterhof and van de Woestijne of the same year (\cite{Win2}). Kononen proved (\cite{KK}) that if $p$ and $r$ are primes such that $p$ is a primitive root modulo $r^m$ for some $m>1$ then 
\begin{equation} \label{kononen1}
g \big( \tfrac{p^{\varphi(r^m)}-1}{r^m},p^{\varphi(r^m)} \big) = \tfrac 12 (p-1)\varphi(r^m)
\end{equation}
where $\varphi$ denotes Euler's totient function. 
If, in addition, $p$ and $r$ are odd primes, then we have
\begin{equation} \label{kononen2}
g \Big(\tfrac{p^{\varphi(r^m)}-1}{2r^m},p^{\varphi(r^m)} \Big) =
\begin{cases}
r^{m-1} \lfloor \tfrac{pr}4 -\tfrac{p}{4r}\rfloor \qquad \quad \text{if } r<p,		\\[1em]
r^{m-1} \lfloor \tfrac{pr}4 -\tfrac{r}{4p}\rfloor \qquad \quad \text{if } r\ge p. 
\end{cases}
\end{equation}
The case $m=1$ is also true and it was more challenging than the general case $m>1$. 
In 2001, Winterhof found a lower bound (see \cite{W}). Some years later, by using some norm techniques, Winterhof and van De Woestijne found the formulas \eqref{kononen1} and \eqref{kononen2} for $m=1$.
Thus, Kononen used the result of the case $m=1$, 
to prove the expressions \eqref{kononen1} and \eqref{kononen2} for any $m\ge 1$. 

We now show how \eqref{kononen1} with $m>1$, 
follows from the case $m=1$ and the reduction formula 
given in Theorem \ref{waringprod main}.

Notice that if $r=2$ , necessarily $m=2$, since there is no primitive roots modulo $2^{m}$ with $m\ge 3$.
This case is exactly Corollary \ref{coro5}, since $p$ is a primitive root modulo $4$ if and only if $p\equiv 3\pmod{4}$.
Suppose now that $p,\,r$ are primes with $r$ odd such that $p$ is a primitive root modulo $r^m$. 
Take 	
	$$b=r^{m-1} \qquad \text{and} \qquad c=r.$$ 
Assume that $m=1$ holds. We have that $c$ is a primitive divisor of $p^{r-1}-1$,
since the Waring number $g(\frac{p^{r-1}-1}{r},p^{r-1})$ exists (see Lemma 3.1 in \cite{PV3}). 

On the other hand, we have that $(\frac{p^{r-1}-1}{r},b)=1$. In fact, if this does not happen then $r$ divides $\frac{p^{r-1}-1}{r}$. In this case, we obtain that $r^{2}\mid p^{r-1}-1$ and thus $p$ is not a primitive element modulo $r^2$. This is a contradiction, since an integer is primitive modulo $r^m$ with $m \ge 2$ if and only if it is primitive modulo $r^2$ 
(when $r$ is an odd prime). 

Notice that $\varphi(rad \, (b)) = r-1$ hence $p^{\varphi(r^m)} = p^{b\varphi(rad \, (b))}$ and thus 
$ \tfrac{p^{\varphi(r^m)}-1}{r^m} = \tfrac{p^{b \varphi(rad \, (b))}-1}{bc}$. 
Therefore, we have
	$$g \Big( \tfrac{p^{\varphi(r^m)}-1}{r^m},p^{\varphi(r^m)} \Big) = 
		g \big( \tfrac{p^{b \varphi(rad \, (b))}-1}{bc}, p^{b \varphi(rad \, (b))} \big) = b g(\tfrac{p^{\varphi(rad \, (b))}-1}{c},p^{\varphi(rad \, (b))})$$
where in the last equality we have used Corollary \ref{fr rad b}. 
Now, we have 
\begin{eqnarray*}
	b g(\tfrac{p^{\varphi(rad \, (b))}-1}{c},p^{\varphi(rad \, (b))}) &=& r^{m-1} g(\tfrac{p^{r-1}-1}{r},p^{r-1}) 
=	r^{m-1} \tfrac{(p-1)(r-1)}{2} = \tfrac 12 (p-1)\varphi(r^m).
\end{eqnarray*}

In the same way, by taking  $b=r^{m-1}$ and $c=2r$, one can obtain \eqref{kononen2} for $m>1$, from the reduction formula  \eqref{g coro} and the case $m=1$, as desired.

\section{Cyclic codes and Waring numbers equal to 2}
Here we  use strongly regular GP-graphs, and a relation with 2-weight cyclic codes, to find new pairs $(k,q)$ such that the corresponding Waring numbers are equal to 2. 
In the next section, we will use these values in combination with the reduction formulas previously obtained to find pairs $(k,q)$ such that $g(k,q)$ is even.

We first give a list of known pairs $(k,q)$ such that $g(k,q)=2$ 
(see for instance \S 2 in \cite{PV3}). 
If $p$ is a prime and $q$ is a prime power we have the following.

\msk

\subsection*{List 1: some known cases of $g(k,q)=2$}
\hrule \vspace{1.5mm}
\begin{enumerate}[$(a)$] \vspace{1.5mm} \hrule  
	\item $g(2,p)=2$ for any odd $p$ and $g(3,p)=2$ for $p\equiv 1 \pmod 3$ with $p\ne 7$ (Small `77, \cite{Sm}). \sk

	\item $g(k,q)=2$ for $2 \le k< \sqrt[4]q +1$ and $k\mid q-1$ (Small `77, \cite{Sm}). In particular, $g(2,q)=2$ for every $q$ odd, since $2\le \sqrt[4]{3}+1$, thus generalizing ($a$). \sk
	
	\item $g(k,q)=2$ if some complicated conditions, depending on the $p$-adic weight of $k$ and $q$, are satisfied (Moreno-Castro `05, Theorem~3.3 and Corollary 3.4 in \cite{MC1}). \sk
	
	\item $g(k,p^{2\ell s})=2$ for $k\mid p^{\ell}+1$ and $s\ne 1$ 
	(Moreno-Castro `08, \cite{MC}, for $k\ne p^{\ell}+1$; Podest\'a-Videla `18, \cite{PV}, for $k=p^{\ell}+1$).
	
	\item $g(\tfrac{p^\ell+1}2, p^{2\ell})=2$ for any $p$ odd (($i$) in Example \ref{otro ejemplo}).
\end{enumerate}
\hrule

\sk
\begin{rem}
($i$) The fact that $g(2,q)=2$ for all odd prime power $q$ can be proved in a graph theoretical way. 
If $\G$ is a graph (directed or undirected), then the diameter $\delta(\G)$ is less than or equal than the number of different eigenvalues of $\G$. In general, the non-principal eigenvalues of $\G(k,q)$ are exactly all the Gaussian periods (see \cite{PV3}). 
It is well known that if $k=2$, there are only 
two such periods and thus $\G(2,q)$ has exactly $3$ different eigenvalues. This implies that $\delta(\G(2,q))\le 2$; but since $\G(2,q)$ is not the complete graph, we must then have that $g(2,q)=\delta(\G(2,q))=2$.

\noindent ($ii$) 
Notice that in ($d$) of List 1, for $k=p^\ell+1$, if $s>2$ then all of the pairs $(k,q)=(p^\ell+1,p^{2\ell s})$ satisfy $k<\sqrt[4]{q}+1$ and so we already know that $g(k,q)=2$ by ($b$) in List~1.
However, if $s=2$, we obtain some new pairs $(k,q)$ satisfying $g(k,q)=2$, since in this case we have that 
$k=\sqrt[4]{q}+1$. This add the limit case in the bound in ($b$) of List 1 for $k=p^\ell+1$.
In other words, all the values 
	$$g(p^\ell+1, p^{4\ell})=2$$ 
for $p$ prime and $\ell\in \N$ are all new and not covered by ($b$) in List 1.
\end{rem}

\begin{exam}
By ($d$) in List 1, we found the values $g(p^\ell+1, p^{2\ell s})=2$ for $s\ge 2$. For instance, for $\ell=1,2,3,4$ and $s=2$ we have the new values 
	$$g(p+1,p^{4}) = g(p^2+1,p^{8}) = g(p^3+1, p^{12}) = g(p^4+1, p^{16}) = 2.$$
Thus, for $p=2,3,5,7$ we respectively have
	\begin{gather*}
	g(3,2^{4}) =g(5, 2^{8}) = g(9,  2^{12}) = g(17, 2^{16}) = 2, \\
	g(4,3^{4}) =g(10,3^{8}) = g(28, 3^{12}) = g(82, 3^{16}) = 2, \\
	g(6,5^{4}) =g(26,5^{8}) = g(126,5^{12}) = g(626,5^{16}) = 2, \\
	g(8,7^{4}) =g(50,7^{8}) = g(344,7^{12}) = g(2402,7^{16}) = 2.
	\end{gather*}	
\end{exam}

\subsubsection*{Strongly regular GP-graphs and cyclic codes}
A \textit{strongly regular graph} (\textit{srg} for short) with parameters $v, \kappa, e, d$, denoted by $srg(v, \kappa, e, d)$, is a 
$\kappa$-regular graph with $v$ vertices such that for any pair of vertices $x,y$ the number of vertices adjacent (resp.\@ non-adjacent) to both $x$ and $y$ is $e\ge 0$ 
(resp.\@ $d\ge 0$). 
Strongly regular graphs are distance regular graphs with diameter $\delta=2$ if $d\ne 0$. Moreover, they are characterized by 
their spectra in the connected case. More precisely, if $\G$ is a connected graph, then $\G$ is a strongly regular graph 
if and only if it has three distinct eigenvalues. The following is straightforward.

\begin{prop} \label{Waringsrg}
	Let $\Gamma = \G(k,q)$ be a generalized Paley graph. If $\Gamma$ is a connected strongly regular graph, then $g(k,q)=2$.
\end{prop}

\begin{proof}
If $\Gamma$ is a connected strongly regular graph, then $\delta(\Gamma)=2$, and by \eqref{g=d} we then have that 
$g(k,p^m)=2$.
\end{proof}

Let $q=p^m$ with $p$ prime and $m \in \N$.  If $k\mid \tfrac{q-1}{p-1}$, the spectrum of $\G(k,q)$ is integral (see Theorem 2.1 in \cite{PV2}). In this case, we have found (see \S5 in \cite{PV2}) that there is a direct relationship between the spectra of GP-graphs $\G(k,q)$ and the weight distribution of the irreducible cyclic $p$-ary codes 
$$\CC(k,q) = \big\{ \big(\Tr_{q/p} (\gamma \omega ^{ki})\big)_{i=0}^{n-1} : \gamma \in \ff_q \big\}, $$ 
where $\omega$ is a primitive element of $\ff_q$ over $\ff_p$ and $n=\frac{q-1}{k}$. More precisely, the eigenvalues and their multiplicities are in a 1-1 correspondence with the weights and their frequencies.

In \cite{SW}, Schmidt and White conjectured that all two-weight irreducible cyclic codes over $\ff_p$, with $k\mid \tfrac{p^m-1}{p-1}$, belong to one of the following disjoint families:

$\bullet$ \textit{Subfield subcodes}: correspond to the powers of the form $k=\frac{p^m-1}{p^{a}-1}$ with $a<m$. 
In this case $\G(k,p^m)$ is not connected and hence we do not consider it. 

$\bullet$ \textit{Semiprimitive codes:} 
correspond to those $k>1$ such that $-1$ is a power of $p$ modulo $k$, that is $k \mid p^t+1$ for some $t$.

$\bullet$ \textit{Exceptional codes}: correspond to the 11 pairs in Table 1 below. 
In this case we say that $(k,p^m)$ is an \textit{exceptional pair}.
\renewcommand{\arraystretch}{1.5}
\begin{table}[H]  \caption{Exceptional pairs} 
\begin{tabular}{||c||c|c|c|c|c|c|c|c|c|c|c||} 
\hline 
$k$ 	& $11$ & $19$ & $35$ & $37$ & $43$ & $67$ & $107$ & $133$ & $163$ & $323$ & $499$ \\ 
\hline
$p^m$   & $3^5$ & $5^9$ & $3^{13}$ & $7^9$ & $11^7$ & $17^{33}$ & $3^{53}$ & $5^{18}$ & $41^{81}$ & $3^{144}$ & $5^{249}$ \\
\hline
\end{tabular}
\end{table}

In terms of graphs, the Schmidt and White's conjecture says that there are only three different kind of GP-graphs which are strongly regular, because their related graphs have only three eigenvalues (two nontrivial eigenvalues which correspond with the nonzero weights of the corresponding code).

We now take a closer look to the semiprimitive case, i.e.\@ when $k>1$ and $q=p^m$ with $-1$ is a power of $p$ modulo $k$. 
Since $k\mid p^t+1$ for some $t$ and $k\mid p^m-1$, we have that $k \mid (p^m-1,p^t+1)$.
It is well-known that if $b$ is an integer then 
$$(b^{m}-1,b^t+1) = \begin{cases}
1 & \qquad \text{if $\frac{m}{(m,t)}$ is odd and $b$ is even}, \\[1mm]
2 & \qquad \text{if $\frac{m}{(m,t)}$ is odd and $b$ is odd},\\[1mm]
b^{(m,t)}+1 & \qquad  \text{if $\frac{m}{(m,t)}$ is even}.
\end{cases}$$ 

Since $k>1$, if $p=2$ we have that $\frac{m}{(m,t)}$ is even. Thus, $m$ is even and $k\mid 2^{\ell}+1$ with $\frac{m}{\ell}$ even ($\ell=(m,t)$). In the same way, 
if $p$ is odd, 
we get that $k=2$ and $m$ even, since $k\mid \frac{p^m -1}{p-1}$ or $k\mid p^{\ell}+1$ with $m$ even such that $\ell\mid m$ with $\frac{m}{\ell}$ even.

Thus, we have proved that if $(k,q)$ is a semiprimitive pair with $q=p^m$, then either: 
\begin{equation} \label{semipr}
\text{$k=2$, $p$ is odd and $m$ is even} \qquad \text{or} \qquad \text{$k \mid p^{\ell}+1$ with $\ell\mid m$ and $\tfrac{m}{\ell}$ even.}
\end{equation}
In any of these cases, if $k\ne p^{\frac m2}+1$ we say that $(k,p^m)$ is a \textit{semiprimitive pair} (see Definition 3.1 in \cite{PV2}).

\sk
The next result gives the Waring numbers, $g(k,q)=2$, associated to all known strongly regular GP-graphs $\G(k,q)$ with $k\mid \tfrac{q-1}{p-1}$, thus extending the cases in List 1.

\begin{thm} \label{srg waring}
	Let $q=p^m$ with $p$ prime and $m\in \N$. Then $g(k,q)=2$ for $(k,q)$ either a semiprimitive pair or an exceptional pair.
\end{thm}

\begin{proof}
First, let $(k,q)$ be a semiprimitive pair.
If $k=2$ and $p$ is odd we know that $g(2,q)=2$ by ($b$) in List 1.
Also, if $k\neq p^{\ell}+1$ and $\ell\neq\frac m2$ 
then $g(k,p^m)=2$ (see \cite{MC}). The case $k=p^{\ell}+1$ with $\ell \neq \frac m2$ was proved in \cite{PV}.
These are the cases covered in ($d$) in List 1.

Now, assume that $m$ is even and $k \mid p^{\frac m2}+1$ with $1< k < p^{\frac m2}+1$.
It is enough to prove that $n=\frac{p^{m}-1}{k}$ is a primitive divisor of $p^m-1$ and, if $p$ is odd, that 
$\tfrac{2}{h}(p^{\frac m2}+1)$ divides $p^{m-1}$. 
Suppose that $p$ is odd and put $k=\tfrac{p^{\frac m2}+1}{h}$. Then, clearly $2k\mid p^m-1$ since 
	$$p^{m}-1 = (p^{\frac m2}-1)(p^{\frac m2}+1) = 2kh (\tfrac{p^{\frac m2}-1}{2}).$$
Thus, it is enough to prove that $n$ is always a primitive divisor of $p^{m}-1$. 
	
Clearly, if $0\le a\le \frac m2$ then $n\nmid p^{a}-1$ since $n > p^{a}-1$ in this case. Suppose now that $\frac m2 < a < m$. Notice that 
	$$(p^{\frac m2}-1,p^{a}-1) = p^{(\frac m2,\,a)}-1$$ 
and thus $n\nmid p^{a}-1$. Therefore, $n \dagger p^m-1$. 
We have that $\G(k,p^m)$ is a connected strongly regular graph and hence, by Proposition \ref{Waringsrg}, $g(k,p^m)=2$. 
	
For the exceptional cases, the multiplicity of the eigenvalue $n=\frac{q-1}k$ is $1$ in all these cases (see \cite{PV2}) and thus $\G$ is connected. Thus, by Proposition~\ref{Waringsrg} we obtain that $g(k,q)=2$ for all the values in Table 1. 
\end{proof}

\begin{exam}[\textit{Exceptional pairs}]
By Theorem \ref{srg waring} and Table 1 we have 
	\begin{gather*}
		g(11,3^5) = g(19,5^9) = g(35, 3^{13}) = g(37, 7^9) = g(43, 11^7) = g(67, 17^{33})=2, \\ 
		g(107, 3^{53})= g(133, 5^{18})= g(163, 41^{81}) = g(323, 3^{144}) = g(499, 5^{249})=2.  
	\end{gather*}
Notice that the only pair $(k,q)$ which does not satisfy $k<\sqrt[4]{q}+1$ is $(11,3^5)$, hence $g(11,3^5)=2$ is a newly obtained exact value. \hfill $\lozenge$
\end{exam}

Now, using known families of semiprimitive pairs $(k,q)$ we give some new explicit Waring numbers of the form $g(k,q)=2$. 

\begin{coro} \label{otro coro}
Let $p$ be a prime and $\ell,s,h \in \N$. Then, we have: 
\begin{enumerate}[$(a)$]
	\item $g(k,p^{2\ell})=2$ for $k=2,3,4$ and $p\equiv -1 \pmod k$, with $\ell \ge 2$ if $p=k-1$. \sk 
	
	\item $g(p^{\ell}+1,p^{2\ell s})=2$ for $s>1$. \sk 
	
	\item $g(\frac{p^{\ell}+1}2, p^{2\ell s})=2$ for $p$ odd. \sk 
	
	\item $g(\frac{p^{\ell}+1}h,p^{2\ell s})=2$ for $\ell$ odd and $p\equiv -1 \pmod h$ with $h>2$. 
\end{enumerate}
\end{coro}

\begin{proof}
For	($a$), one can check that $(2,p^{2\ell})$ with $p$ odd, $(3, p^{2\ell})$ with $p\equiv 2 \pmod 3$ where $\ell\ge 2$ if $p=2$, and $(4, p^{2\ell})$ with $p\equiv 3 \pmod 4$ where $\ell \ge 2$ if $p=3$, are all semiprimitive pairs (they were previously obtained in Example 3.2 of \cite{PV2}). By Theorem \ref{srg waring} we get $g(k,p^{2\ell})=2$ in all these cases.
	
For ($b$)--($d$), it is clear that $(p^\ell+1, p^{2\ell s})$ is semiprimitive and hence $(\frac{p^{\ell}+1}h, p^{2\ell s})$ 
is also semiprimitive for any $h\ge 2$. Thus, the result follows directly from Theorem~\ref{srg waring}. The extra hypothesis on $p$ are necessary for $k$ to be an integer.  
\end{proof}

Notice that, in the previous corollary, the pair $(2,p^{2\ell})$ in ($a$) gives the classic Paley graphs $\G(2,p^{2\ell})$.  Also, ($b$) is ($d$) in List 1 obtained in a different form, and ($c$) generalizes ($e$) in List 1. 

One can check that almost all the semiprimitive pairs given in Corollary \ref{otro coro} satisfy Small's result given in ($b$) of List 1, with the only exceptions of the pair in ($b$) with $s=2$, the pair in ($c$) with $s=1$ and $\ell >1$, and ($d$) with $s=1$ and suitable choices of $p$ and $\ell$ in this case. 
We now give a whole family of new Waring numbers equal to 2, which cannot be obtained from Small's result.

\begin{prop} \label{otra prop}
Let $p$ be a prime and $h, \ell \in \N$. For every $h \mid p^{\ell}+1$, $h>1$, we have    
\begin{equation} \label{g=2 no small}
g(\tfrac{p^{\ell }+1}h,p^{2\ell })=2.
\end{equation}
Moreover, if $1< h \le p^{\frac{\ell }2}-1$, the pairs $(\tfrac{p^{\ell }+1}h,p^{2\ell})$ do not satisfy Small's condition \eqref{small g2}. 
\end{prop}

\begin{proof}
It is clear that the pair $(\frac{p^{\ell}+1}h,p^{2\ell })$ is semiprimitive for any $h>1$. Thus, \eqref{g=2 no small} follows directly from Theorem~\ref{srg waring}. 
The last statement is straightforward by comparing $k=\tfrac{p^{\ell }+1}h$ with $\sqrt[4]{q}+1=p^{\frac{\ell }{2}}+1$ in each case.
\end{proof}

\begin{rem}
Notice that all semiprimitive pairs $(k,q)$ that do not satisfy Small's condition \eqref{small g2} are exactly those already mentioned in the above proposition. 
In fact, suppose $(k,q)$ is a semiprimitive pair, with $q=p^m$, not satisfying Small's condition. Then, $k\ne 2$, $k \ge \sqrt[4]{q}+1$ and by \eqref{semipr} there exists $\ell\mid m$ such that  $k\mid p^{\ell}+1$ and $\frac m\ell$ is even. 
Thus, if $\frac m\ell >4$, the pair $(k,q)$ clearly satisfies Small's condition. Hence, $\frac m\ell$ is $2$ or $4$. 
For $\frac m\ell =2$, we are in the hypothesis of the above proposition.
For $\frac m\ell = 4$, the unique pair $(k,q)$ not satisfying Small's condition is $k=p^{\ell}+1$ and $q=p^{4\ell}$. 
However, the pair $(p^{\ell}+1,p^{4\ell})$, i.e.\@ $s=2$ in $(b)$ of Corollary \ref{otro coro}, can also be obtained from \eqref{g=2 no small} by taking $h=p^{\frac{\ell}2}-1$ and $\ell$ even. 
In this way, we have found all possible Waring numbers $g(k,q)=2$ that can be obtained by way of semiprimitive pairs without using Small's result, hence basically new. 

\end{rem}

\section{Even values of Waring numbers}
In the previous section we collect known values $g(k,q)=2$ from other works and from this one obtained by using the reduction formula in Section 3 (see List 1). Furthermore, we also give some new pairs $(k,q)$ with  $g(k,q)=2$, that we obtained using  a relation with semiprimitive 2-weight cyclic codes $\mathcal{C}(k,q)$. Also, Theorem \ref{teo 42} with $t\ge 1$ exhibits a general result giving even Waring numbers, although not explicit ones. 
	
Here, we will obtain infinitely many explicit even values of Waring numbers, not previously known or covered by 
Theorem \ref{teo 42}. The general strategy that we will use is to begin with our reduction formula $$g(\tfrac{p^{ab}-1}{bc},p^{ab}) = b g(\tfrac{p^a-1}{c}, p^{a})$$ 
from Theorem \ref{waringprod main} (just taking $b=2$ here is the result in Corollary \ref{coro1}) 
combined with particular instances where we know for sure that $g(\tfrac{p^a-1}{c}, p^{a})=2$, hence getting 
$$g(\tfrac{p^{ab}-1}{bc},p^{ab}) = 2b.$$
We will use both semiprimitive pairs $(u,p^a)$ and Small's result in \eqref{small g2}.  

We begin by recollecting known even Waring numbers $g(k,q)=2b$ with $b>1$.

\subsection*{List 2: some known cases of $g(k,p)=2b$, $b>1$, with $p$ prime}
\hrule \vspace{1.5mm}
\begin{enumerate}[$(a)$] \vspace{1.5mm} \hrule  
	\item In the table given in Remark 7.3.50 of Mullen-Panario's book \cite{MP} there are some small known values of $g(k,p)=4$ from different sources. \msk
	
	\item Let $p$ be a prime and $a,b,c,d \in \N$ (unique if exist) such that $a^2+b^2+ab=p$ with $a>b$ and $c^2+d^2=p$ with $c>d$. Cochrane and Pinner showed in \cite{CP} that $g(\tfrac{p-1}3,p)=a+b-1$ and $g(\tfrac{p-1}4,p)=c-1$. These Waring numbers are even provided that $a+b$ is odd or $c$ is odd, respectively. \msk 
\end{enumerate}
\hrule

\msk 
	
\subsection*{List 3: some known cases of $g(k,q)=2b$, $b>1$, with $q$ a prime power}
\hrule \vspace{1.5mm}
\begin{enumerate}[$(a)$] \vspace{1.5mm} \hrule  

	\item In Example \ref{g=4}, we get $g(\tfrac 32 (p^2+1), p^4)=g(4(p^2+1), p^4)=4$ for any odd prime $p$, 
	where the second equality holds if further $p\equiv \pm 3 \pmod 8$ and $p\ge 67$. \msk 
	
	\item In ($iii$) of Example \ref{otro ejemplo} we get $g(\tfrac 23 (7^{2a}+7^a+1), 7^{3a})=6$ for every $a\in \N$. \msk 
	 
	\item In Corollary \ref{coro4} we obtain $g(\tfrac{p^{b}-1}{b^2}, p^{b})= 2b$ for $b$ odd prime and $p$ prime of the form $p=2b+1$. \msk 
	
	\item In Corollary \ref{coro5} we get $g(\tfrac{p^{2}-1}{4}, p^{2})=p-1$ for any odd prime $p$ with
	$p\equiv 3\pmod{4}$. \msk 
	
	\item In Corollary \ref{coro2} we get $g(\tfrac{2(p^{ab}-1)}{b(p^a-1)}, p^{ab})=2b$ for $p$ and $b$ different odd primes and $b\mid p^a-1$. \msk 
	  
	\item Proposition \ref{red pow} with $r=2$ and $p^a \equiv 1 \pmod 4$ gives
	$$g \big( \tfrac{1}{2^t} \{ p^{(2^t-1)a}+\cdots +p^{4a}+p^{2a}+1 \}, p^{2^ta} \big ) = 2^tg(\tfrac{p^a-1}{c}, p^a)$$
	for every $p$ odd prime, $c\dagger p^a-1$ and $t\ge 2$. A more general result is obtained in Theorem \ref{teo 42}.
	   
	\item Kononen's result $g(\tfrac{p^{\varphi(b)}-1}{b},p^{\varphi(b)})=\tfrac 12 (p-1) \varphi(b)$ with $b=r^m$ given in \eqref{kononen1} gives even values $g(k,q)\ge 4$ for $p$ and $r$ distinct primes with $p$ primitive root modulo $r^m$ for some $m>1$, since the Euler function takes even values for $b>2$.  \msk

	\item In Theorem 6.1 ($b$) of \cite{PV3} we showed that if $p$ and $r$ are different odd primes then $g(\tfrac{p^{ab}-1}{b(p^a-1)},p^{ab})=2r$ for every $a\in \N$ provided that $(p^a,r)=1$ and $p^a\equiv \pm 1\pmod r$. \msk 
	
	\item In Corollary 6.13 in \cite{PV3} we showed that for any $p$ prime and $a, t \in \N$ one has
	$$g\big( \tfrac{1}{2^t}(p^{(2^t-1)a} +\cdots +p^{2a}+p^a+1), p^{2^t a} \big)=2^t.$$ 
\end{enumerate}
\hrule

\

Next, we give a simple criterion on $a,b,c \in \N$ to have even Waring numbers of the form $g(\tfrac{p^{ab}-1}{bc},p^{ab})=2b$.

\begin{prop} \label{teo par}
Let $p$ be a prime and $a,b,u \in \N$ such that $p\nmid b$ and $u\mid p^{a}-1$ with $u>1$. Put $c=\frac{p^a-1}{u}$. If $(u,p^a)$ is a semiprimitive pair or else if $2\le u\le p^{\frac a4}-1$, then 
we have the following cases:
\begin{enumerate}[$(a)$]
	\item If $u$ is odd and $p^{a}\equiv 1\pmod{4}$, then for all $t\in \mathbb{N}$ we have
		$$g(\tfrac{p^{2^t a}-1}{2^t c},p^{2^t a}) = g(\tfrac{u}{2^t}(p^{(2^t-1)a} +\cdots +p^{2a}+p^a+1), p^{2^t a}) = 2^{t+1}.$$

	\item If $(b,u)=1$ and $\varphi(rad(b))\mid a$ and $b>1$ is not a power of $2$, then 
		$$g(\tfrac{p^{ab}-1}{bc},p^{ab})= 2b.$$	
\end{enumerate}
\end{prop}

\begin{proof}
First, notice that by Theorem \ref{srg waring} and Small's result (see ($b$) in List 1 in \S 6), we have that $g(u,p^a)=2$ either if $(u,p^a)$ is semiprimitive or if $2\le u\le p^{\frac a4}-1$.  
The existence of these Waring numbers implies that $c \dagger p^a-1$, in both cases. 

To see item ($a$), we will apply Proposition \ref{red pow} with $b=2^t$ ($r=2$), since in this case we have $(p,b)=1$ 
and $c\dagger p^a-1$. 
It is enough to see that $(b,u)=1$, but this holds since $b$ is a power of $2$ and $u$ is odd by hypothesis. 
Now, the result follows directly by Proposition~\ref{red pow}.

Item ($b$) is a direct consequence of ($a$) in Corollary \ref{fr rad b}, since all the required hypothesis of this result hold. 
\end{proof}

Notice that in item ($a$) of the above result, we can always choose an odd integer $u$ of the form $\frac{p^{a}-1}{c}$. 
In fact, since given $p$, we can take any $a$ such that $p^{a}-1$ is not a power of $2$. 
Now, if we take $c=2^{v_2(p^{a}-1)}$ it is clear that $u=\tfrac{p^{a}-1}{c}>1$ is odd.

\begin{rem} \label{casos fav par}
In Proposition \ref{teo par}, if $b=r^t$ with $r$ an odd prime and 
$a=2s\ell=r-1$ we get the following:

\noindent ($i$) 
Assume $s>1$. If $r \nmid p^\ell+1$ and $u\mid p^{\ell}+1$ with $u>1$ then, by ($b$) of Proposition~\ref{teo par}, for all 
$t\in \mathbb{N}$ we have
\begin{equation} \label{pow phi}
	g(\tfrac{u(p^{\varphi(r^{t+1})}-1)}{r^t(p^{r-1}-1)},p^{\varphi(r^{t+1})})=2r^t. 
\end{equation}
In particular, \eqref{pow phi} holds if $p$ is primitive modulo $r$ with $r>3$.
	
\noindent ($ii$)                                                                              
If $s=1$ and $r\nmid u$ with $u>1$ a proper divisor of $p^{\frac{r-1}{2}}+1$, then by ($b$) of Proposition~\ref{teo par},
\eqref{pow phi} holds for every $t\in \N$. 
\end{rem}

\begin{exam}
$(i)$ Let $r=5$ in $(i)$ of Remark \ref{casos fav par}. Then, $r-1=4=2s\ell$ and in this case, we necessarily 
have that $s=2$ and $\ell=1$, since $s>1$.
Now, let $p$ be a prime such that $5\nmid p+1$. 
For $u\mid p+1$ with $u>1$, by $(i)$ of the previous remark, for all $t\in \N$ we have that
	$$g(\tfrac{u(p^{\varphi(5^{t+1})}-1)}{5^t(p^{4}-1)},p^{\varphi(5^{t+1})})=2 \cdot5^t. $$ 
For instance, if $p=2$ and $t=1$, then necessarily $u=3$ (since $u>1$) and hence we have $g(41{.}943, 1{.}048{.}576)=10$. 
 
\noindent $(ii)$ 
Let $r=7$ in $(ii)$ of Remark \ref{casos fav par}, then $\tfrac{r-1}{2}=3$. If $p=3$ and $h$ is a proper divisor of $p^{\frac{r-1}{2}}+1=28$, thus $h\in \{2,4,7,14\}$.
 Since $7\nmid u=\frac{p^{3}+1}{h}$, then for $u\in \{2,4\}$ we have that
 $$g(\tfrac{u(3^{\varphi(7^{t+1})}-1)}{7^t(3^6-1)},3^{\varphi(7^{t+1})})=2\cdot 7^t$$ 
for all $t \in \N$. For instance,
$g(42{.}943{.}088{.}356{.}166{.}546, 109{.}418{.}989{.}131{.}512{.}359{.}209)=
14$ is obtained by taking $t=1$ and $u=2$ in the above expression.
\hfill $\lozenge$
\end{exam}

\begin{exam}
	Let $p=5$, $u=\frac{p^{a}-1}{c}=4$ and $b=3$. If $a=8$, then the pair $(4,5^{24})$ is not semiprimitive but we have that 
	$4<5^{\frac{a}{4}}-1=24$ and $(u,b)=1$. Thus, 
	$$g(\tfrac{4(5^{24}-1)}{3(5^8-1)}, 5^{24})=g(203{.}451{.}041{.}668, \, 59{.}604{.}644{.}775{.}390{.}625)=6$$ 
	by the last theorem. \hfill $\lozenge$
\end{exam}

\subsubsection*{Waring numbers $g(k,q)=4$}
In Section 6 we study Waring numbers $g(k,q)=2$. We now give some new Waring numbers $g(k,q)=4$.
We will use known results in Lists 2 and 3, and the results in this section.

\subsection*{List 4: some cases of $g(k,q)=4$}
\hrule \vspace{1.5mm}
\begin{enumerate}[$(a)$] \vspace{1.5mm} \hrule  
\item From ($a$) in List 2 we have all the pairs $(k,q)$ such that $g(k,q)=4$ for $2\le k\le 11$. These pairs $(k,q)$ are given by $(6,p)$ and $(7,p)$ with $p\le 29$ prime, $(6,31)$, $(7,43)$, $(8,41)$, $(10,41)$, $(10,61)$ and $(11,89)$. \msk 

\item From ($b$) in List 2, we only get $g(7,29)=g(10,41)=4$, which are covered in $(a)$. 
\msk 

\item From Kononen's result in ($f$) of List 3, we see that we can only take $p=3$ and $\varphi(b)=4$ (hence $b=5,8,10,12$) or $p=5$ and $\varphi(b)=2$ (hence $b=3,4$). This gives the values $g(6,25)=g(8,25)=4$ and $g(5,81)=g(8,81)=g(10,81)=g(12,81)=4$. \msk  

\item From $(a)$ in List 3, for $p$ an odd prime we have that $g(\tfrac 32 (p^2+1), p^4)=4$ and also that $g(4(p^2+1), p^4)=4$ 
if further $p\equiv \pm 3 \pmod 8$ and $p\ge 67$. \msk 

\item From ($i$) in List 3, we have $g\big(\tfrac 14 (p^{3a}+p^{2a}+p^a+1), p^{4a} \big)=4$ for every $a\in \N$. \msk

\item By ($a$) in Proposition \ref{teo par}, and under its hypothesis, we have 
$g(\tfrac{u}2(p^{a}-1),p^{2a})=4$.
\end{enumerate}
\hrule

\msk 

Notice that our results $(d)$--$(f)$ provide infinite families of Waring numbers equal to 4.

\subsubsection*{Data Availability Statement:} 
Data sharing not applicable to this article as no datasets were generated or analyzed during the current study.


\begin{thebibliography}{XXX}
	\bibitem{BH} \textsc{A.E.\@ Brouwer, W.H.\@ Haemers}. 
	\textit{Structure and uniqueness of the $(81,20,1,6)$ strongly regular graph}.
	Discrete Mathematics \textbf{106/107} (1992) 77--82. 
	
	
	\bibitem{Ci}{\sc J.\@ A.\@ Cipra}. 
	\textit{Waring's number in a finite field}. 
	Integers \textbf{9:4} (2009), 435--440.
	
	
	\bibitem{CiCP}{\sc J.\@ A.\@ Cipra, T.\@ Cochrane, C.\@ Pinner}. 
	\textit{Heilbronn's conjecture on Waring's number (mod $p$)}. 
	J.\@ Number Theory \textbf{125:2} (2007), 289--297.
	
	
	\bibitem{CP}{\sc T.\@ Cochrane, C.\@ Pinner}. 
	\textit{Sum-product estimates applied to Waring's problem mod $p$}. 
	Integers \textbf{8:1} (2008), A46.
			
	
 	\bibitem{Do}{\sc M.\@ M.\@ Dodson}. 
 	\textit{On Waring’s problem in $GF[p]$}, Acta Arith.\@ \textbf{19} (1971), 147--173.
	
	
	\bibitem{DT}{\sc M.\@ M.\@ Dodson, A.\@ Tietäväinen}. 
	\textit{A note on Waring's problem in $GF[p]$}, Acta Arith.\@ \textbf{30} (1976), 159--167. 
	
	
	\bibitem{GS}{\sc C.\@ Garcia, P.\@ Sol\'e}. 
	\textit{Diameter lower bound for Waring graphs and multiloop networks}. 
	Discrete Mathematics \textbf{111} (1993) 257--261.


	\bibitem{GK}{\sc  D.\@ Ghinelli, J.D.\@ Key}. 
	\textit{Codes from incidence matrices and line graphs of Paley graphs}.
	Adv.\@ Math.\@ Comm.\@ \textbf{5} (2011) 93--108.
	
	
	\bibitem{Gl}{\sc A.\@ A.\@ Glibichuk}. 
	\textit{Sums of powers of subsets of an arbitrary finite field}. 
	Izvestiya: Mathematics \textbf{75:2} (2011), 253.
	
	
	\bibitem{GlR}{\sc A.\@ Glibichuk, M.\@ Rudnev}. 
	\textit{On additive properties of product sets in an arbitrary finite field}. 
	Journal d'Analyse Math\'ematique \textbf{108:1} (2009), 159--170.
	
	
	\bibitem{HIK}{\sc R.\@ Hammack, W.\@ Imrich, S.\@ Klav\v{z}ar}. 
	\textit{Handbook of product graphs}. 
	CRC Press 2nd Ed.\@ (2011).
	

	\bibitem{HL}\textsc{H.\@ G.\@ Hardy, J.\@ E.\@ Littlewood} 
	\textit{Some problems of 'Partitio Numerorum' (VIII): The number $\Gamma(k)$ in Waring's Problem.}
	Proc.\@ London Math.\@ Soc.\@ \textbf{28:7} (1928), 518--542.


	\bibitem{J}{\sc G.\@ Jones}.
	\textit{Characterisations and Galois conjugacy of generalised Paley maps.}. 
	J.\@ Comb.\@ Theory, Ser.\@ B \textbf{103:2} (2013) 209--219.
	
	\bibitem{JW}{\sc G.\@ Jones, J.\@ Wolfart}.
	\textit{Dessins d'Enfants on Riemann Surfaces}. 
	Springer International Publishing Switzerland, (2016).
	
	
	\bibitem{KL}{\sc J.D.\@ Key, J.\@ Limbupasiriporn}. 
	\textit{Partial permutation decoding for codes from Paley graphs}.
	Cong.\@ Numer.\@ \textbf{170} (2004) 143--155.
	
	
	\bibitem{KK}{\sc K.\@ Kononen}.
	\textit{More exact solutions to Waring's problem for finite fields}. 
	Acta Arith.\@ \textbf{145} (2010) 209--212.
		
		
	\bibitem{Kon}{\sc S.\@ V.\@ Konyagin}. 
	\textit{On estimates of sums of Gauss and Waring's problem for prime module}.
	Proc.\@ Steklov Inst.\@ Math.\@ \textbf{198} (1994), 105--117.
		
	
	\bibitem{LP}{\sc T.K.\@ Lim, C.\@ Praeger}. 
	\textit{On Generalised Paley Graphs and their automorphism groups}.
	Michigan Math.\@ J.\@ \textbf{58} (2009) 294--308.
	
	
	\bibitem{MC1} \textsc{O.\@ Moreno, F.N.\@ Castro}.
	\textit{On the calculation and estimation of Waring number for finite fields}.
	S\'eminaires et Congr\`es SMF \textbf{11} (2005) 29--40.


	\bibitem{MC}\textsc{O.\@ Moreno, F.N.\@ Castro}.
	\textit{Optimal divisibility for certain diagonal equations over finite fields}.
	J.\@ Ramanujan Math.\@ Soc.\@ \textbf{23} (2008) 43--61.
	
		
	\bibitem{MP}{\sc G.L.\@ Mullen, D.\@ Panario}. 
	\textit{Handbook of finite fields}. 
	CRC Press, 2013.
	
	
	\bibitem{PP}{\sc G.\@ Pearce, C.\@ Praeger}. 
	\textit{Generalised Paley graphs with a product structure}. 
	Ann.\@ Comb.\@ \textbf{23} (2019) 171--182.
	
	
	\bibitem{PV} \textsc{R.A.\@ Podest\'a, D.E.\@ Videla}.
	\textit{The spectra of generalized Paley graphs of $q^\ell+1$-th powers and applications}.
	arXiv:1812.03332 (2018). 
	
	
	\bibitem{PV2} \textsc{R.A.\@ Podest\'a, D.E.\@ Videla}.
	\textit{Spectral properties of generalized Paley graphs and their associated irreducible cyclic codes}.
	arXiv:1908.08097  (2019). 
		
	\bibitem{PV3} \textsc{R.A.\@ Podest\'a, D.E.\@ Videla}.
	\textit{The Waring's problem over finite fields through generalized Paley graphs}.
	Discrete Math.\@ \textbf{344} (2021) 112324. 
	
	\bibitem{PV4} \textsc{R.A.\@ Podest\'a, D.E.\@ Videla}.
	\textit{The weight distribution of irreducible cyclic codes associated with decomposable generalized Paley graphs}.
	Adv.\@ Math.\@ Comm.\@  (2021), in press.
	
	
	\bibitem{Ro}\textsc{K.H.\@ Rosen}. 
	\textit{Elementary Number Theory and Its Applications}.
	 $4$th ed., Addison Wesley, Reading, MA, 2000.
	
	
	\bibitem{SW}
	\textsc{B.\@ Schmidt, C.\@ White}.
	\textit{All two weight irreducible cyclic codes}.
	Finite Fields App.\@ \textbf{8} (2002) 1--17.
	
	
	\bibitem{SL}{\sc P.\@ Seneviratne, J.\@ Limbupasiriporn}. 
	\textit{Permutation decoding from generalized Paley graphs}. 
	Appl.\@ Algebra in Eng.\@ Comm.\@ and Computing \textbf{24} (2013) 225--236.
	
		
	\bibitem{Sm}{\sc C.\@ Small}. 
	\textit{Sum of powers in large finite fields}. 
	Proc.\@ Amer.\@ Math.\@ Society \textbf{65:1} (1977), 356--359.
	
	
	\bibitem{Ti}{\sc A.\@ Tietäväinen}. 
	\textit{On diagonal forms over finite fields}, Ann.\@ Univ.\@ Turku Ser.\@ A \textbf{118} (1968), 10~pp. 
	
	
	\bibitem{To}{\sc L.\@ Tornheim}. 
	\textit{Sums of $n$-th powers in fields of prime characteristic}, Duke Math.\@ J.\@ \textbf{4} (1938), 359--362.
	
	
	\bibitem{V} \textsc{D.E.\@ Videla}.
	\textit{On diagonal equations over finite fields via walks in NEPS of graphs}.
	Finite Fields App.\@ \textbf{75} (2021) 101882. 


	\bibitem{Win} \textsc{A.\@ Winterhof}.
	\textit{On Waring's problem in finite fields}.
	Acta Arith.\@ \textbf{87} (1998) 171--177.
	
	
	\bibitem{W} \textsc{A.\@ Winterhof}.
	\textit{A note on Waring's problem in finite fields}.
	Acta Arith.\@ \textbf{96:4} (2001), 365--368.
	
	
	\bibitem{Win2} \textsc{A.\@ Winterhof, C.\@ van de Woestijne}.
	\textit{Exact values to Waring's problem in finite fields}.
	Acta Arith.\@ \textbf{141} (2010) 171--190.
\end{thebibliography}
\end{document}